\documentclass[11pt]{article}
\usepackage{amsfonts}
\usepackage{graphics}
\usepackage{indentfirst}
\usepackage{color}
\usepackage{cite}
\usepackage{latexsym}
\usepackage[paper=a4paper, left=1.6cm, right=1.6cm, top=1.8cm, bottom=1.6cm, headheight=5.5pt, footskip=0.8cm, footnotesep=0.8cm, centering, includefoot]{geometry}
\usepackage{amsmath}
\allowdisplaybreaks
\usepackage{amssymb}
\usepackage[colorlinks, linkcolor=red]{hyperref}
\hypersetup{urlcolor=red, citecolor=red}
\usepackage[dvips]{epsfig}
\usepackage{amscd}

\usepackage{amsthm}
\usepackage{mathrsfs}
\usepackage{verbatim}
\newtheorem{Theorem}{Theorem}[section]
\newtheorem{Lemma}{Lemma}[section]

\newtheorem{Definition}{Definition}[section]
\newtheorem{Remark}{Remark}[section]
\newtheorem{Proposition}{Proposition}[section]

\DeclareMathOperator{\divv}{div}

\providecommand{\norm}[1]{\left\Vert#1\right\Vert}

\newcounter{RomanNumber}

\def\be{\begin{equation}}
\def\en{\end{equation}}
\def\bs{\begin{split}}
\def\es{\end{split}}

\makeatletter
\@addtoreset{equation}{section}
\makeatother
\makeatletter
\@addtoreset{equation}{section}
\makeatother

\title{Long-time behavior to the 3D isentropic compressible Navier--Stokes equations
\thanks{Wu's research was partially supported by Natural Science Foundation of
Fujian Province and Fujian Alliance of Mathematics (Nos. 2022J01304 and 2023SXLMMS08). Zhong's research was partially supported by Fundamental Research Funds for the Central Universities (No. SWU--KU24001) and National Natural Science Foundation of China (No. 12371227).}
}

\author{Guochun Wu$\,^{\rm 1}\,$,\
Xin Zhong$\,^{\rm 2}\,$ {\thanks{E-mail addresses:
guochunwu@126.com (G. Wu), xzhong1014@amss.ac.cn (X. Zhong).}}\date{}\\
\footnotesize $^{\rm 1}\,$
School of Mathematics and Statistics, Xiamen University of Technology, Xiamen 361024, P. R. China\\
\footnotesize $^{\rm 2}\,$ School of Mathematics and Statistics, Southwest University, Chongqing 400715, P. R. China}

\begin{document}
\maketitle

\begin{abstract}
We are concerned with the long-time behavior of classical solutions to the isentropic compressible Navier--Stokes equations in $\mathbb R^3$. Our main results and innovations can be stated as follows:
\begin{itemize}
\item Under the assumption that the density $\rho({\bf{x}}, t)$ verifies $\rho({\bf{x}},0)\geq c>0$ and $\sup_{t\geq 0}\|\rho(\cdot,t)\|_{L^\infty}\leq M$, we establish the optimal decay rates of the solutions. This greatly improves the previous result (Arch. Ration. Mech. Anal. 234 (2019), 1167--1222), where the authors require an extra hypothesis $\sup_{t\geq 0}\|\rho(\cdot,t)\|_{C^\alpha}\leq M$ with $\alpha$ arbitrarily small.
\item We prove that the vacuum state will persist for any time provided that the initial density contains vacuum and the far-field density is away from vacuum, which extends the torus case obtained in (SIAM J. Math. Anal. 55 (2023), 882--899) to the whole space.
\item We derive the decay properties of the solutions with vacuum as far-field density. This in particular gives the first result concerning the $L^\infty$-decay with a rate $(1+t)^{-1}$ for the pressure to the 3D compressible Navier--Stokes equations in the presence of vacuum.
\end{itemize}
The main ingredient of the proof relies on the techniques involving blow-up criterion, a key time-independent positive upper and lower bounds of the density, and a regularity interpolation trick.
\end{abstract}



\tableofcontents

\section{Introduction}

\subsection{Background and motivation}

We study the Cauchy problem of the isentropic compressible Navier--Stokes equations in $\mathbb{R}^3$:
\begin{align}\label{1.1}
\begin{cases}
\rho_{t}+\divv(\rho\mathbf{u})=0,\\
(\rho\mathbf{u})_{t}+\divv(\rho\mathbf{u}\otimes\mathbf{u})
-\mu\Delta\mathbf{u}-(\mu+\lambda)\nabla\text{div}\mathbf{u}+\nabla P(\rho)=0,
\end{cases}
\end{align}
with the initial condition
\begin{align}\label{1.2}
(\rho,\rho\mathbf{u})(\mathbf{x},0)
=(\rho_0,\rho_0\mathbf{u}_0)({\bf x}),\ \ {\bf x}\in\mathbb{R}^3,
\end{align}
and the far-field behavior
\begin{align}\label{1.3}
(\rho(\mathbf{x},t), \mathbf{u}(\mathbf{x},t))\rightarrow(\tilde{\rho},\mathbf{0}),\ \ \text{as}\ |\mathbf{x}|\rightarrow\infty,\ t>0,
\end{align}
for a non-negative constant $\tilde{\rho}$. Here $\rho$, $\mathbf{u}=(u^1,u^2,u^3)$, and $P=\rho^\gamma\ (\gamma>1)$ denote the density, velocity, and pressure of the fluid, respectively. The constants $\mu$ and $\lambda$ are the shear viscosity and the bulk viscosity of the fluid satisfying the physical restrictions
\begin{align*}
\mu>0\ \ \text{and}\ \ 2\mu+3\lambda\geq 0.
\end{align*}

There is a huge literature on the global existence (well-posedness) of solutions to the multi-dimensional compressible isentropic Navier--Stokes equations. When the initial density is away from a vacuum, the global well-posedness of strong solutions in three-dimensional (3D) domains was first obtained by Matsumura and Nishida \cite{Matsumura3,Matsumura2}. With the help of the \textit{effective viscous flux}, Hoff \cite{Hoff1,Hoff2} proved the global existence of non-vacuum solutions with discontinuous initial data, while Danchin \cite{Dan1} investigated the global existence and uniqueness of strong solutions in the critical Besov spaces (see also \cite{Dan3,CMZ,H11} for the $L^p$ setting). It should be pointed out that the results stated above require smallness of the initial oscillation on the perturbation of initial data near the non-vacuum equilibrium state. This assumption concerning the small oscillations on the initial perturbation of a non-vacuum state is recently removed by Huang--Li--Xin \cite{HuangLi1}, where they proved the global existence and uniqueness of classical solutions with smooth initial data that are of small energy in $\mathbb{R}^3$ (see also \cite{CL23} for the case of 3D bounded domains). The major breakthrough on the solutions with large initial data and vacuum is due to Lions \cite{Lions}, where he used the renormalization skills introduced by DiPerna and Lions \cite{DiPerna1} to establish global weak solutions in $\mathbb{R}^n\ (n=2,3)$ provided that $\gamma\geq\frac{3n}{n+2}$. Later on, Feireisl--Novotn\'y--Petzeltov\'a \cite{Feireisl} improved Lions' result to the case $\gamma>\frac{n}{2}$. At the same time, Jiang and Zhang \cite{Jiang-Zhang 2001, Jiang1} derived global weak solutions for any $\gamma>1$ when the initial data are assumed to have some spherically symmetric or axisymmetric properties. Nonetheless, due to the possible concentration of finite kinetic energy in very small domains, it still seems to be a challenge to show the global existence of weak solutions with general 3D data for $\gamma \in (1,\frac{n}{2}]$. There are also very interesting investigations about global large strong solutions in some sense for multi-dimensional compressible isentropic Navier--Stokes equations, please refer to \cite{DM17,DM23,HPZ24,FZ18,ZLZ20} and references therein.

It is also of great interest to study the large-time behavior of solutions. The readers can refer to \cite{Dan4,Duan1,HZ95,KK02,KK05,LiuWang,Matsumura3,Matsumura2} and references therein for large-time behavior of global smooth solutions to the compressible Navier--Stokes system under the smallness assumption on the initial perturbation. Some important progress has been made about long time behavior of large solutions for multi-dimensional compressible Navier--Stokes equations by many authors. Feireisl and Petzeltov\'a \cite{Feireisl2} first showed that any weak solution converges to a fixed stationary state as time goes to infinity via the weak convergence method. Under the assumptions that the density is essentially bounded and has uniform-in-time positive lower bound, Padula \cite{Padula} proved that weak solutions decay exponentially to the equilibrium state in $L^2$-norm. By using the Bogovskii operator and constructing a suitable Lyapunov functional, Peng--Shi--Wu \cite{PSW} improved the result in \cite{Padula} to the case that they didn't need the time-independent upper and lower bounds of density. He--Huang--Wang \cite{He} investigated the global stability of large strong solutions to the 3D Cauchy problem. More precisely, under the hypothesis that the density $\rho({\bf x},t)$ verifies $\inf_{{\bf x}\in\mathbb R^3}\rho_0({\bf x})\ge c_0>0$ and $\sup_{t\ge0}\|\rho(t)\|_{C^\alpha}\le M$ with arbitrarily small $\alpha$, they established the convergence of the solutions to its associated equilibrium states with an explicit decay rate which is the same as that of the heat equation. It should be emphasized that the assumption $\sup_{t\ge0}\|\rho(t)\|_{C^\alpha}\le M$ plays an essential role in their analysis (see the proof of Proposition 2.3 in \cite{He} for details).

As mentioned in many papers (refer to \cite{HuangLi1,CL23,Desjardins,Sun,Hu,Hu1,Merle1,Merle2} for instance), one of the main difficulties for the stability of solutions to the compressible isentropic Navier--Stokes equations is the concentration of the density. In particular, Desjardins \cite{Desjardins} studied the regularity of weak solutions for small time under periodic boundary conditions, and particularly showed that weak solutions in $\mathbb T^2$ turn out to be smooth as long as the density remains bounded in $L^\infty(\mathbb T^2)$. Sun--Wang--Zhang \cite{Sun} proved that the concentration of the density must be responsible for the loss of the regularity of the strong solution in finite time. Hu \cite{Hu,Hu1} investigated the concentration phenomenon of weak solutions with finite energy, and showed that the concentration set has a parabolic Hausdorff dimension related to $\gamma$ by a provided formula. More recently, Merle--Rapha\"el--Rodnianski--Szeftel \cite{Merle1,Merle2} constructed a set of finite energy smooth initial data for which the corresponding solutions to the compressible 3D Navier--Stokes and Euler equations implode (with infinite density) at a later time at a point, and completely described the associated formation of singularity. Then a natural question, but of importance and interest, is that can we show the long-time behavior of global classical solutions to the system \eqref{1.1} with {\it bounded density}?
The main goal of the present paper is to give an affirmative response to this problem.

\subsection{Main results}

Before stating our main result, we first explain the notations and conventions used throughout this paper. We denote by $C$ a positive generic constant which may vary at different places. The norms in Sobolev spaces
$H^\ell(\mathbb{R}^3)$ and $W^{\ell, r}(\mathbb{R}^3)$ are denoted
respectively by $\|\cdot\|_{H^\ell}$ and $\|\cdot\|_{W^{\ell, r}}$ for $\ell\geq0$
and $r\geq1$. Particularly, for $\ell=0$, we will simply use
$\|\cdot\|_{L^r}$. For the sake of conciseness, we
do not precise in functional space names when they are concerned
with scalar--valued or vector--valued functions. $\|(f, g)\|_X$
denotes $\|f\|_X+\|g\|_X$.   We denote
$\nabla=\partial_x=(\partial_1,\partial_2,\partial_3)$, where
$\partial_i=\partial_{x_i}$, $\nabla_i=\partial_i$ and put
$\partial_x^\ell f=\nabla^\ell f=\nabla(\nabla^{\ell-1}f)$. We denote $\|u\|_{D^{k,r}}\overset{\triangle}=\|\nabla^k u\|_{L^r}$, $D^k(\mathbb R^3)=D^{k,2}(\mathbb R^3)$, and $D^1(\mathbb R^3)=\{u\in L^6\big | \|\nabla u\|_{L^2}<\infty\}$.  Let $\Lambda^s$ be
the pseudo differential operator defined by
\begin{equation}\Lambda^sf=\mathfrak{F}^{-1}(|{\bf \xi}|^s\hat f),~\hbox{for}~s\in \mathbb{R},\nonumber\end{equation}
where both $\hat f$ and $\mathfrak{F}(f)$ denote the Fourier transform of $f$. The homogenous
Sobolev space $\dot{H}^s(\mathbb{R}^3)$ is endowed with the norm
$\|f\|_{\dot{H}^s}:=\|\Lambda^s f\|_{L^2}$. We also drop the domain $\mathbb R^3$ in integrands over $\mathbb R^3$, that is,  $\int\cdot \mathrm{d}{\bf x}=\int_{\mathbb{R}^3}\cdot \mathrm{d}{\bf x}$. Introducing the Leray projector $\mathcal P\triangleq \text{Id}+\nabla(-\Delta)^{-1}\divv$ onto divergence-free vector fields, and $\mathcal Q\triangleq \text{Id}-\mathcal P$. In particular, because
$\mathcal P$ and $\mathcal Q$ are smooth homogeneous of degree $0$ Fourier multipliers, they map $L^p$ into itself for any $1<p<\infty$.

 We also recall the Littlewood--Paley decomposition. Choose a radial function $\varphi\in S(\mathbb{R}^{3})$ supported in $\mathcal {C}=\{\xi\in
\mathbb{R}^{N},\frac{3}{4}\leq|\xi|\leq\frac{8}{3}\}$ such that
$$\sum\limits_{q\in \mathbb{Z}}\varphi(2^{-q}\xi)=1,\hbox{\ \ for all\ \ } \xi\neq 0.$$
For $q\in \mathbb{Z}$, we define the following dyadic blocks:
$$\Delta_q f=\mathfrak{F}^{-1}(\varphi(2^{-q}\xi)\mathfrak{F}f).$$
We denote the space $\mathcal {D}^{'}(\mathbb{R}^{N})$ by the dual
space of $\mathcal {D}(\mathbb{R}^{N})=\{f\in
S(\mathbb{R}^{3});D^{\alpha}\hat{f}(0)=0;\forall \alpha\in
\mathbb{N}^{3} \hbox{multi-index}\}$, which also can be identified by
the quotient space of $S^{'}(\mathbb{R}^{3})/\mathcal {P}$ with
polynomials space $\mathcal {P}$. The formal equality
$$f=\sum\limits_{q\in \mathbb{Z}}\Delta_q f$$
holds true for $f\in \mathcal {D}^{'}(\mathbb{R}^{3})$ and is called
the homogeneous Littlewood--Paley decomposition.

\begin{Definition}Let $s\in \mathbb{R},1\leq p,r\leq +\infty$. The
homogeneous Besov space $\dot{B}_{p,r}^{s}$ is defined by
$$\dot{B}_{p,r}^{s}(\mathbb{R}^{3})=\{f\in D^{'}(\mathbb{R}^{3}):\|f\|_{\dot{B}_{p,r}^{s}}<+\infty\},$$
where
$$\|f\|_{\dot{B}_{p,r}^{s}}\triangleq\|2^{qs}\|\Delta_q
f\|_{L^{p}}\|_{l^{r}}.$$
\end{Definition}

As in \cite{Hoff1}, the \textit{effective viscous flux} $F$
and vorticity ${\bf w}$ are defined by
\begin{align}\label{1.4}
F\overset{\text{def}}= (2\mu+\lambda)\divv{\bf u}-(P(\rho)-P(\tilde{\rho}))\ \ \text{and} \ \ {\bf w}\overset{\text{def}}=\nabla\times {\bf u}.
\end{align}
By \eqref{1.1}$_2$, one has that
\begin{align}\label{1.5}
\Delta F=\divv(\rho\dot{{\bf u}})\ \ \text{and} \ \ \mu\Delta  {\bf w}
=\nabla\times(\rho\dot {\bf u}),
\end{align}
where $``\ \dot\ \ "$ denotes the material derivative which is defined by
$$\dot f\overset{\text{def}}=f_t+{\bf u}\cdot\nabla f.$$
The potential energy $G(\rho)$ is defined by
\begin{align}\label{1.6}
G(\rho)=\rho\int^\rho_{{\tilde\rho}}\frac{P(s)-P(\tilde\rho)}{s^2}\mathrm{d}s,
\end{align}
 it is not hard to deduce that there exists a positive constant $C$
depending only on $\tilde\rho$ and $\bar\rho$ such that
\begin{align}\label{1.7}
\begin{cases}
G(\rho)=\frac{1}{\gamma-1}P,\ & \text{if} \ \tilde\rho=0,\\
\frac{1}{C(\tilde\rho,\bar\rho)}(\rho-\tilde\rho)^2\leq G(\rho)\leq C(\tilde\rho,\bar\rho)(\rho-\tilde\rho)^2,\ &\text{if}\ \tilde\rho>0\ \text{and}\ \rho\le \bar\rho.
\end{cases}
\end{align}

Our first main novelty of this paper is to establish the optimal convergence rates of the solutions to the Cauchy problem \eqref{1.1}--\eqref{1.3}.

\begin{Theorem}\label{thm1.1}
Let $\tilde\rho=1$. Assume that the initial data $(\rho_0,{\bf u}_0)$ satisfies
\begin{align*}
\inf\rho_0\ge c_0>0,\ \ K\triangleq\|(\rho_0-1,{\bf u}_0)\|_{H^3}<\infty.
\end{align*}
If $2\mu>\lambda$ and $(\rho,{\bf u})$ is a global classical solution to the Cauchy problem \eqref{1.1}--\eqref{1.3} verifying that
\begin{align}\label{1.8}
\sup\limits_{t\ge 0}\|\rho(\cdot,t)\|_{L^\infty}\leq M
\end{align}
for some positive constant $M$, then the following properties hold.

(1) {\bf (Lower bound for the density)} There exists a positive constant $\underline{\rho}=\underline{\rho}(c_0,M,K)$ such that
\begin{align}\label{z1.8}
\rho(t)\ge \underline{\rho},\ \ \text{for all}\ \ t\ge 0.
\end{align}

(2) {\bf (Uniform-in-time bound for the solution)} For all $t\ge 0$, it holds that
\begin{align}\label{1.9}
\|(\rho-1,{\bf u})(t)\|_{H^3}^2
+\int_0^t\big(\|\nabla \rho(\tau)\|_{H^2}^2
+\|\nabla{\bf u}(\tau)\|_{H^3}^2\big)\mathrm{d}\tau
\le C(c_0,M,K).
\end{align}

(3) {\bf (Decay estimate for the solution)}  If additionally
$N_0\triangleq\|(\rho_0-1,{\bf u}_0)\|_{\dot B^{-s}_{2,\infty}}<\infty$ with $0<s\le \frac{3}{2}$, we have
\begin{align}\label{1.10}
\|\nabla^k (\rho-1,{\bf u})(t)\|_{H^{3-k}}\le C(c_0,M,K,N_0)(1+t)^{-\frac{k+s}{2}t},\ \ k\in\{0,1,2\}.
\end{align}
\end{Theorem}

\begin{Remark}
Theorem \ref{thm1.1} gives the global dynamics of the system \eqref{1.1}
when the initial data is far away from the equilibrium in the whole space.
This is in sharp contrast to the results of \cite{Dan4,KK02,KK05,Matsumura3,Matsumura2} where the small initial perturbation near the non-vacuum equilibrium state is used to get the long-time behavior of the solution.
\end{Remark}

\begin{Remark}
It should be noticed that in Theorem \ref{thm1.1}, the density is only assumed to be bounded from above, while the global dynamics of large solutions in \cite[Theorem 1.1]{He} requires an extra hypothesis $\sup_{t\ge0}\|\rho(t)\|_{C^\alpha}\le M$ with arbitrarily small $\alpha$, which plays an essential role in establishing the uniform positive lower bound of the density (see the proof of Proposition 2.3 in \cite{He} for details). Thus, we greatly improve the main result of \cite[Theorem 1.1]{He}.
\end{Remark}

Next, we shall prove that the vacuum states will not vanish for any time provided that the vacuum states are present initially and the far-field density is away from a vacuum.

\begin{Theorem}\label{thm1.2}
Let $\tilde\rho=1$. Suppose that the initial data $(\rho_0,{\bf u}_0)$ satisfies
\begin{align}
(\rho_0|{\bf u}_0|^2+G(\rho_0))\in L^1(\mathbb R^3),\ \ {\bf u}_0\in D^1(\mathbb R^3)\cap D^3(\mathbb R^3),\label{1.11}\\
(\rho_0-1,P(\rho_0)-P(1))\in H^3, \ \ 0=\inf\rho_0\le \sup\rho_0\le \bar\rho,\label{1.12}
\end{align}
and the compatibility condition
\begin{align}\label{1.13}
-\mu\Delta{\bf u}_0-(\mu+\lambda)\nabla\divv {\bf u}_0+\nabla P(\rho_0)=\rho_0 g
\end{align}
for some $g\in D^1(\mathbb R^3)$ with $\sqrt{\rho_0}g\in L^2(\mathbb R^3)$.
Let $2\mu>\lambda$ and $(\rho,{\bf u})$ be a global classical solution to the Cauchy problem \eqref{1.1}--\eqref{1.3} verifying \eqref{1.8}, then it holds that
\begin{align}\label{1.14}
\lim\limits_{t\rightarrow\infty}\|\rho-1\|_{L^q}=0,\ \ \text{for any}\ \ q>2,
\end{align}
and
\begin{align}\label{1.15}
\inf\limits_{{\bf x}\in\mathbb R^3}\rho({\bf
x},t)=0,\ \ \text{for any}\ \ t\ge 0.
\end{align}
\end{Theorem}

\begin{Remark}
\eqref{1.15} shows that the vacuum state will not vanish for any time, which together with the large-time behavior \eqref{1.14} leads to the gradient of the density has to blow up as time goes to infinity, in the sense that for any $r>3$,
\begin{align*}
\lim\limits_{t\rightarrow\infty}\|\nabla\rho(\cdot,t)\|_{L^r}=\infty.
\end{align*}
This is in sharp contrast to that in \cite{Dan1,Matsumura3} where the gradient of the density is suitably small uniformly for all time.
\end{Remark}

\begin{Remark}
Theorem \ref{thm1.2} extends the torus case \cite{WLZ} to the whole space. However, this is a non-trivial exercise since the Poincar\'e inequality is crucial in the arguments in \cite{WLZ} which does not hold in the whole space.
\end{Remark}


Finally, we have the following results concerning the decay properties of the global classical solutions when the far-field density is vacuum (i.e., $\tilde\rho=0$).

\begin{Theorem}\label{thm1.3}
Assume that $\tilde\rho=0$, $\rho_0\in L^1(\mathbb R^3)$, and the initial data $(\rho_0,{\bf u}_0)$ satisfies \eqref{1.11}--\eqref{1.13}. Let $7\mu>\lambda$ and $(\rho,{\bf u})$ be a global classical solution to the Cauchy problem \eqref{1.1}--\eqref{1.3} verifying \eqref{1.8}, then it holds that, for any $t\geq0$,
\begin{align}\label{1.16}
\norm{P(\cdot,t)}_{L^p}\le C(1+t)^{-1+\frac{1}{p\gamma}},\ \ \text{for any}\ 1\leq p\leq\infty,
\end{align}
and
\begin{align}\label{1.17}
\int\rho|{\bf u}|^2 \mathrm{d}{\bf x}
\le
C(1+t)^{-1-\frac{1}{3\gamma}}.
\end{align}
\end{Theorem}

\begin{Remark}
To the best of our knowledge, the $L^\infty$-decay with a rate $(1+t)^{-1}$ for the pressure in \eqref{1.16} is completely new to the 3D compressible Navier--Stokes equations \eqref{1.1} in the presence of vacuum.
\end{Remark}

\subsection{Strategy of the proofs}

We shall now explain the major difficulties and techniques in the proof of our main results. It should be pointed out that the methods used in \cite{PSW,WLZ} cannot be adopted here to prove Theorems \ref{thm1.1} and \ref{thm1.2} since their arguments rely more or less on the fact that domain is bounded. For instance, the Poincar\'e inequality and the $L^p$ bound of Bogovskii operator are used which do not hold in the whole space. Moreover, as stated above, the analysis in \cite{He} relies heavily on the condition that the density is bounded uniformly in time in the H\"older space $C^\alpha$ with $\alpha$ arbitrarily small. Consequently, we cannot simply employ the frameworks of \cite{PSW,WLZ,He}, but need some new observations and ideas to establish the desired {\it a priori} estimates under the assumption \eqref{1.8}.

We first show the time-independent positive lower bound \eqref{z1.8} of the density. Motivated by \cite{Desjardins}, we rewrite the mass equation \eqref{1.1}$_1$ in terms of $\ln\rho$, which immediately yields the time-dependent positive lower bound of $\rho$ (see \eqref{wz5}). Then, arguing by contradiction, and making fully use of \eqref{1.1}$_1$ and Lagrangian coordinates, we derive a crucial bound on $\rho-1$ for large time (see \eqref{3.10}). This together with the dissipative estimate on the effective viscous flux (see \eqref{3.3}) leads to the time-independent positive lower bound of $\rho$ immediately. Such bound is important to show the large-time behavior of $\|\rho(t)-1\|_{L^\infty}$ (see \eqref{wz6}). Next, we attempt to prove the uniform-in-time bound \eqref{1.9} of the solution. It turns out that one of the key issues is to derive the time-independent estimate on $L^6$-norm of the gradient of $\rho$. Multiplying the equation of $\nabla\rho$ (see \eqref{3.14}) by $|\nabla\rho|^4\nabla\rho$ yields that the time-dependent estimates on both the $L^\infty_tL^6_x$-norm of the density gradient and the $L^2_tL^\infty_x$-norm of the velocity gradient can be obtained by solving a logarithm Gronwall inequality based on Lemma \ref{lemma2.5}. Then, after some careful analysis based on contradiction arguments, we succeed in obtaining that, for large time, the $L^6$-norm of the gradient of the density can be controlled with the help of an interpolation trick and Proposition \ref{pro3.1}. This combined with \eqref{wz6} leads to the large-time behavior of $\|\nabla\rho(t)\|_{L^6}$ and the time-independent estimate on the $L^2_tL^\infty_x$-norm of the velocity gradient. With this key bound for $L^\infty(0,\infty;W^{1,6})$-norm of the density at hand, we can deduce \eqref{1.9} (see the detailed proof of Lemma \ref{lemma3.3}). Then, using these {\it a priori} estimates, one derives the decay estimate \eqref{1.10} by applying the arguments as in \cite{GuoY}. Next, we show Theorem \ref{thm1.2}. To do so, the key step is to establish that $\|\rho-1\|_{L^6}$ converges to zero as time goes to infinity (see \eqref{wz8}), which combined with \eqref{3.3}, the upper bound of the density, and the interpolation inequality yields the desired results \eqref{1.14} and \eqref{1.15}.

To prove Theorem \ref{thm1.3}, the key ingredient is to establish an algebraic decay rate of the total energy (see \eqref{5.1}). By the basic energy estimate
and making fully use of the momentum equation \eqref{1.1}$_2$, one can obtain an energy-dissipation inequality of the form
\begin{align*}
\frac{\mathrm{d}}{\mathrm{d}t}\mathcal M_1(t)+C(\mathcal M_1(t))^\frac{2\gamma-1}{\gamma-1}\le 0,
\end{align*}
where $\mathcal M_1(t)$ is equivalent to the total energy. This implies the algebraic decay estimate in Lemma \ref{lemma5.1} immediately. Then, based on such estimate and the large-time asymptotic behavior of solutions in \cite{LiXin} (see \eqref{5.13} and \eqref{5.14}), we are able to derive the desired decay estimates \eqref{1.16} and \eqref{1.17}.

\subsection{Structure of the paper}

In Section \ref{sec2}, we collect some elementary facts and inequalities that will be used later. Section \ref{sec3} is devoted to the {\it a priori} estimates, which play important roles in the proof of Theorems \ref{thm1.1} and \ref{thm1.2} in the next section. Finally, we give the proof of Theorem \ref{thm1.3} in Section \ref{sec5}.

\section{Preliminaries}\label{sec2}

In this section, we collect some known facts and elementary inequalities that will be  used frequently later.

We start with the local existence and uniqueness of classical solutions when the initial density may vanish in an open set, whose proof can be performed by using standard procedures as in \cite{ChoKim}.

\begin{Lemma} For $\tilde\rho\ge 0$, assume that the initial data satisfy \eqref{1.11}--\eqref{1.13}. Then there exist a small time $T_*$ and a unique classical solution $(\rho,{\bf u})$ to the Cauchy problem \eqref{1.1}--\eqref{1.3} on $\mathbb R^3\times (0,T_*]$ such that
\begin{align}\label{2.1}
\begin{cases}
(\rho-\tilde\rho,P(\rho)-P(\tilde\rho))\in C([0,T_*];H^3(\mathbb R^3)),\\
{\bf u}\in C([0,T_*];D^1(\mathbb R^3)\cap D^3(\mathbb R^3))\cap L^2((0,T_*);D^4(\mathbb R^3)),\\
{\bf u}_t\in L^\infty((0,T_*);D^1(\mathbb R^3))\cap L^2((0,T_*);D^2(\mathbb R^3)),\\
 \sqrt{\rho}{\bf u}_{tt}\in L^2((0,T_*);L^2(\mathbb R^3)),\  \sqrt{t}{\bf u}\in L^\infty((0,T_*);D^4(\mathbb R^3)),\\
  \sqrt{t}{\bf u}_{t}\in L^\infty((0,T_*);D^2(\mathbb R^3)),\  \sqrt{t}{\bf u}_{tt}\in L^2((0,T_*);D^1(\mathbb R^3)),\\
    t^\frac{1}{2}\sqrt{\rho}{\bf u}_{tt}\in L^\infty((0,T_*);L^2(\mathbb R^3)),\  t{\bf u}_{t}\in L^\infty((0,T_*);D^3(\mathbb R^3)),\\
    t{\bf u}_{tt}\in L^\infty((0,T_*);D^1(\mathbb R^3))\cap L^2((0,T_*);D^2(\mathbb R^3)),\\
    t\sqrt{\rho}{\bf u}_{ttt}\in L^2((0,T_*);L^2(\mathbb R^3)),\ t^\frac{3}{2}{\bf u}_{tt}\in L^\infty((0,T_*);D^2(\mathbb R^3)),\\
  t^\frac{3}{2}{\bf u}_{ttt}\in L^2((0,T_*);D^1(\mathbb R^3)),\ t^\frac{3}{2}\sqrt{\rho}{\bf u}_{ttt}\in L^\infty((0,T_*);L^2(\mathbb R^3)).
\end{cases}
\end{align}
\end{Lemma}

Next, the following estimates follow from \eqref{1.4} and the standard $L^p$-estimate of elliptic system, cf. \cite{HuangLi1}.
\begin{Lemma}\label{Lemm2.2} Let $(\rho,{\bf u})$ be as in Theorems \ref{thm1.1}--\ref{thm1.3}, there exists a generic positive constant $C$ such that, for $p\in[2,6]$,
\begin{align}\label{2.2}
\|\nabla F\|_{L^p}+\|\nabla{\bf w}\|_{L^p}&\leq C\|\rho\dot {\bf u}\|_{L^p},
\\ \label{2.3}
\|{F}\|_{L^p}+\|{\bf w}\|_{L^p}&\leq C\|\rho\dot {\bf u}\|_{L^2}^\frac{3p-6}{2p}\big(\|\nabla{\bf u}\|_{L^2}+\|P-P(\tilde\rho)\|_{L^2}\big)^\frac{6-p}{2p},
\\ \label{2.4}
\|\nabla {\bf u}\|_{L^p}&\leq C\big(\|F\|_{L^p}+\|{\bf w}\|_{L^p}+\|P-P(\tilde\rho)\|_{L^p}\big),
\\ \label{2.5}
\|\nabla {\bf u}\|_{L^p}&\leq C\|\nabla{\bf u}\|_{L^2}^\frac{6-p}{2p}\big(\|\rho\dot{\bf u}\|_{L^2}+\|P-P(\tilde\rho)\|_{L^6}\big)^\frac{3p-6}{2p}.
\end{align}
\end{Lemma}

We recall the well-known Gagliardo--Nirenberg inequality (see \cite[Theorem 12.87]{LG}).
 \begin{Lemma}\label{1interpolation}
 (Gagliardo--Nirenberg inequality, general case). Let $1 \leq$ $p, q \leq \infty, m \in \mathbb{N}, k \in \mathbb{N}_0$, with $0 \leq k<m$, and let $a$, $r$ be such that
\begin{align*}
0 \leq a \leq 1-\frac{k}{m}
\end{align*}
and
\begin{align*}
(1-a)\Big(\frac{1}{p}-\frac{m-k}{3}\Big)
+a\Big(\frac{1}{q}+\frac{k}{3}\Big)
=\frac{1}{r}\in[0,1].
\end{align*}
Then there exists a constant $C=C(m, p, q, a, k)>0$ such that
\begin{align}\label{2.6}
\left\|\nabla^k f\right\|_{L^r} \leq C\|f\|_{L^q}^a\left\|\nabla^m f\right\|_{L^p}^{1-a}
\end{align}
for every $f \in L^q\left(\mathbb{R}^3\right) \cap \dot{W}^{m, p}\left(\mathbb{R}^3\right)$(where the homogeneous Sobolev space $\dot{W}^{m, p}\left(\mathbb{R}^3\right)$ is the space of all functions $f\in  L_{\text{loc}}^1(\mathbb{R}^3)$ whose $\alpha$-th weak derivative $\partial^\alpha f$ belongs to $L^p(\mathbb{R}^3)$
with $|\alpha|=m$), with the following exceptional cases:

\vspace{0.1cm}

(i) If $k=0, m p<3$, and $q=\infty$, we assume that $f$ vanishes at infinity.

\vspace{0.2cm}

(ii) If $1<p<\infty$ and $m-k-\frac3p$ is a nonnegative integer, then \eqref{2.6} only holds for $0<a \leq 1-\frac{k}{m}$.
\end{Lemma}

 We also  need the following  embedding and interpolation estimates on  Besov spaces, cf. \cite{Sohinger1}.
\begin{Lemma}\label{lemma2.3}
(1) Let $0< s\leq\frac{3}{2}$ and $1\leq p< 2$ satisfying $\frac{1}{2}+\frac{s}{3}=\frac{1}{p}$, then there exists a constant $C=C(s, p)>0$ such that
\begin{align}\label{2.7}
\|f\|_{\dot B_{2,\infty}^{-s}}\leq C\|f\|_{L^p}.
\end{align}
(2) Let $s>0$ and $\ell\geq 0$, then there exists a constant $C=C(s, \ell)>0$ such that
\begin{align}\label{2.8}
\|\nabla^\ell f\|_{L^2}\leq C\|\nabla^{\ell+1}f\|_{L^2}^{1-\alpha} \|f\|_{\dot B_{2,\infty}^{-s}}^\alpha,\ \ \text{with}\ \alpha=\frac{1}{\ell+1+s}.
\end{align}
\end{Lemma}

Finally, the following Beale--Kato--Majda type inequality will be used to estimate $\|\nabla\varrho\|_{L^6}$ and $\|\nabla{\bf u}\|_{L^\infty}$, whose proof can be found in \cite[Lemma 2.3]{HLZ11}.
\begin{Lemma}\label{lemma2.5}
Suppose that $3<q<\infty$, then there is a positive constant $C$ depending only on $q$ such that, for any $\nabla{\bf u} \in L^2(\mathbb R^3)\cap D^{1,q}(\mathbb R^3)$,
\begin{align}\label{2.9}
\|\nabla{\bf u}\|_{L^\infty}\le C\big(\|\divv{\bf u}\|_{L^\infty}+\|{\bf w}\|_{L^\infty}\big)\ln\big(\text{e}+\|\nabla^2{\bf u}\|_{L^q}\big)+C\|\nabla{\bf u}\|_{L^2}+C.
\end{align}
\end{Lemma}

\section{A priori estimates}\label{sec3}

In this section, we will establish some necessary {\it a priori} bounds for classical solutions $(\rho,\mathbf{u})$ to the Cauchy problem \eqref{1.1}--\eqref{1.3} with $\tilde\rho=1$. In what follows, let $\varrho\overset{\text{def}}=\rho-1$ and $\sigma\overset{\text{def}}=P(\rho)-P(\tilde\rho)=\rho^\gamma-1$. Noting that
\begin{align*}
\sigma=\varrho\int_0^1\gamma(\theta\rho+(1-\theta))^{\gamma-1}\mathrm{d}\theta,
\end{align*}
hence, if $\rho\le M$, one has that
\begin{align}\label{3.1}
|\sigma|\sim |\varrho|.
\end{align}
Setting
\begin{align*} H(\rho)=\frac{1}{2\mu+\lambda}\Big[\frac{\gamma^2}{2(2\gamma-1)}(\rho-1)^2-\Big(\frac{\gamma-1}{2(2\gamma-1)}\rho^\gamma
+\frac{\gamma(\gamma-1)}{2(2\gamma-1)}\rho-\frac{\gamma^2+2\gamma-1}{2(2\gamma-1)}\Big)G(\rho)\Big],
\end{align*}
due to \eqref{1.7} and $\rho\le M$, we get that
\begin{align}\label{3.2}
|H(\rho)|\le CG(\rho).
\end{align}

The following Proposition \ref{pro3.1} deals with the uniform-in-time bounds and the dissipation inequality, whose proof can be found in \cite[Proposition 2.2]{He}.

\begin{Proposition}\label{pro3.1}
Let $2\mu>\lambda$, $\tilde\rho=1$, and $(\rho,{\bf u})$ be a classical solution of the Cauchy problem \eqref{1.1}--\eqref{1.3} with $0\le \rho\le \bar\rho$, then there exist constants $A_i\ (i=1,\cdots,6)$ such that
\begin{align}\label{3.3}
A_1&\|(\rho^\frac{1}{4}{\bf u})(t)\|^4_{L^4}+A_2\Big(\mu\|\nabla{\bf u}(t)\|_{L^2}^2+(\mu+\lambda)\|\divv{\bf u}(t)\|_{L^2}^2
-\int (\sigma\divv {\bf u})(t)\mathrm{d}{\bf x}+\int H(\rho)(t)\mathrm{d}{\bf x}\Big)\nonumber
\\&+A_3\|\sigma(t)\|_{L^6}^2+A_4\Big(\int G(\rho)(t)\mathrm{d}{\bf x}+\|\sqrt{\rho}{\bf u}(t)\|_{L^2}^2\Big)+A_5\|(\sqrt{\rho}\dot{\bf u})(t)\|_{L^2}^2\nonumber
\\&+A_6\int_0^t\big(\||{\bf u}||\nabla{\bf u}|\|_{L^2}^2+\|\nabla{\bf u}\|_{L^2}^2+\|\sqrt{\rho}{\bf u}_t\|_{L^2}^2+\|\nabla {\bf w}\|_{L^2}^2
+\|\nabla F\|_{L^2}^2+\|\sigma\|_{L^6}^2\nonumber
\\&+\|\nabla{\bf u}\|^2_{L^6}+\|\nabla\dot{\bf u}\|^2_{L^2}+\| F\|_{W^{1,6}}^2+\| {\bf w}\|_{W^{1,6}}^2\big)\mathrm{d}\tau\le C(M,K),
\end{align}
where
\begin{align}\label{3.4}
A_1&\|\rho^\frac{1}{4}{\bf u}\|^4_{L^4}+A_2\Big(\mu\|\nabla{\bf u}\|_{L^2}^2+(\mu+\lambda)\|\divv{\bf u}\|_{L^2}^2-\int \sigma\divv {\bf u}\mathrm{d}{\bf x}+\int H(\rho)\mathrm{d}{\bf x}\Big)\nonumber
\\&+A_3\|\sigma\|_{L^6}^2+A_4\Big(\int G(\rho)\mathrm{d}{\bf x}+\|\sqrt{\rho}{\bf u}\|_{L^2}^2\Big)+A_5\|\sqrt{\rho}\dot{\bf u}\|_{L^2}^2\nonumber
\\ \sim&\  \|\rho-1\|_{L^2}+\|\sqrt{\rho}{\bf u}\|_{L^2}^2+\|\rho^\frac{1}{4}{\bf u}\|^4_{L^4}+\|\nabla{\bf u}\|_{L^2}^2+\|\rho-1\|_{L^6}
+\|\sqrt{\rho}\dot{\bf u}\|_{L^2}^2.
\end{align}
\end{Proposition}

We have the following uniform positive lower bound of $\rho$.

\begin{Lemma}\label{lemma3.1}
Under the assumptions of Theorem \ref{thm1.1}, there exists a
 positive constant $c_1$ independent of $t$ such that
\begin{align}\label{3.5}
\inf\limits_{{\bf x}\in\mathbb R^3}\rho({\bf x},t)\ge c_1,
\end{align}
and
\begin{align}\label{3.6}
\lim\limits_{t\rightarrow\infty}\|\rho(\cdot,t)-1\|_{L^\infty}=0.
\end{align}
\end{Lemma}
\begin{proof} Let ${\bf y}\in\mathbb
R^3$ and define the corresponding particle path ${\bf x}(t,{\bf y})$ by
\begin{align*}
\begin{cases}
\frac{ \mathrm{d}{\bf x}(t,{\bf y})}{\mathrm{d}t}={\bf u}({\bf x}(t,{\bf y}),t),\\
{\bf x}(t_0,{\bf y})={\bf y},
\end{cases}
\end{align*}
then we rewrite the mass equation \eqref{1.1}$_1$ as
\begin{align*}
\frac{\mathrm{d}}{\mathrm{d}t}\ln \rho=-\divv {\bf u},
\end{align*}
which together with the definition of $F$ in \eqref{1.4}, \eqref{1.8}, and \eqref{3.3} implies that
\begin{align}\label{wz5}
\rho(t,{\bf x})\ge \rho_0\text{e}^{{-\int_0^t}\|\divv {\bf u}\|_{L^\infty}\mathrm{d}\tau}\ge c_0\text{e}^{{-C\int_0^t}\|(\sigma,F)\|_{L^\infty}\mathrm{d}\tau}\ge c_0\text{e}^{-C(t+1)}.
\end{align}
Moreover, due to \eqref{3.3}, for any $\delta>0$, there exists a constant $T_1>0$ such that, for any $t\ge T_1$,
\begin{align}\label{3.7}
&\int_t^\infty\big(\|\nabla{\bf u}\|_{L^2}^2+\|\sqrt{\rho}{\bf u}_t\|_{L^2}^2+\|\nabla {\bf w}\|_{L^2}^2+\|\nabla F\|_{L^2}^2+\|\sigma\|_{L^6}^2+\|\nabla{\bf u}\|^2_{L^6}+\|\nabla\dot{\bf u}\|^2_{L^2}\nonumber\\&\ \ \ +\| F\|_{W^{1,6}}^2+\| {\bf w}\|_{W^{1,6}}^2\big)\mathrm{d}\tau\le \delta.
\end{align}

Next, we argue by contradiction to show \eqref{3.5}. Indeed, if \eqref{3.5} is false, then there exists a time $T_2\ge T_1$ such that
\begin{align}\label{wz1}
0<\tilde{c}\triangleq\rho({\bf x}(T_2),T_2)\le \Big(\frac{1}{2}\Big)^\frac{1}{\gamma}.
\end{align}
Choosing a minimal value of $T_3>T_2$ such that $\rho({\bf x}(T_3),T_3)=\frac{\tilde{c}}{2}$, so one has that
\begin{align}\label{wz2}
\rho({\bf x}(t),t)\in[\tilde{c}/2,\tilde{c}]\ \ \text{for}\ \ t\in[T_2,T_3].
\end{align}
By the definition of  $F$ in \eqref{1.4}, we rewrite \eqref{1.1}$_1$ as
\begin{align}\label{3.8}
\frac{\mathrm{d}}{\mathrm{d}t}(\rho-1)+\frac{1}{2\mu+\lambda}\rho(\rho^\gamma-1)
=-\frac{1}{2\mu+\lambda}\rho F.
\end{align}
Integrating \eqref{3.8} along particle
trajectories from  $T_2$ to $T_3$,
abbreviating $\rho({\bf x}(t),t)$ by $\rho(t)$ for convenience, we get that
\begin{align}\label{3.9}
(2\mu+\lambda)(\rho-1)\Big{|}^{T_3}_{T_2}+\int^{T_3}_{T_2}
\rho(\rho^\gamma-1)\mathrm{d}\tau=-\int^{T_3}_{T_2}\rho F({\bf x}(\tau),\tau)\mathrm{d}\tau.
\end{align}
By \eqref{wz1} and \eqref{wz2}, one has that
\begin{align}\label{3.10}
(2\mu+\lambda)(\rho-1)\Big{|}^{T_3}_{T_2}+\int^{T_3}_{T_2}\rho(\rho^\gamma-1)
\mathrm{d}\tau
&\le -\frac{\tilde{c}(2\mu+\lambda)}{2}+\int^{T_3}_{T_2}\frac{\tilde{c}}{2}
\Big(\frac{1}{2}-1\Big)\mathrm{d}\tau\nonumber
\\&= -\frac{\tilde{c}(2\mu+\lambda)}{2}-\frac{\tilde{c}}{4}(T_3-T_2).
\end{align}
In virtue of \eqref{3.7} and \eqref{wz2}, we deduce that
\begin{align}\label{3.11}
-\int^{T_3}_{T_2}\rho F({\bf x}(\tau),\tau)\mathrm{d}\tau
\ge -\tilde{c}\int^{T_3}_{T_2} \|F\|_{L^\infty}\mathrm{d}\tau
\ge -C\tilde{c}\int^{T_3}_{T_2} \|F\|_{W^{1,6}}\mathrm{d}\tau
\ge -C\tilde{c}\delta^\frac{1}{2}(T_3-T_2)^\frac{1}{2}.
\end{align}
As a result, substituting \eqref{3.10} and \eqref{3.11} into \eqref{3.9}, we derive that
\begin{align*}
-\frac{\tilde{c}(2\mu+\lambda)}{2}-\frac{\tilde{c}}{4}(T_3-T_2)\ge -C\tilde{c}\delta^\frac{1}{2}(T_3-T_2)^\frac{1}{2},
\end{align*}
which is impossible if $\delta$ is small enough. Thus, the desired \eqref{3.5} holds.

Now we turn to prove \eqref{3.6}. Multiplying \eqref{3.8} by $\rho-1$, we arrive at
\begin{align*}
\frac{\mathrm{d}}{\mathrm{d}t}(\rho-1)^2+\frac{1}{2\mu+\lambda}\rho(\rho^\gamma-1)(\rho-1)
=-\frac{1}{2\mu+\lambda}\rho (\rho-1)F.
\end{align*}
By \eqref{3.1} and \eqref{3.5}, there exists a positive constant $c_2$ such that
\begin{align*}
\frac{1}{2\mu+\lambda}\rho(\rho^\gamma-1)(\rho-1)\ge 2c_2(\rho-1)^2.
\end{align*}
Then, by Cauchy--Schwarz inequality, we conclude that
\begin{align}\label{wz6}
\frac{\mathrm{d}}{\mathrm{d}t}(\rho-1)^2+c_2(\rho-1)^2\le C\|F\|_{L^\infty}^2\le C\|F\|_{W^{1,6}}^2,
\end{align}
which together with Gronwall inequality shows that
\begin{align*}
(\rho(t)-1)^2&\le \text{e}^{-c_2t}(\rho(0)-1)^2+C\int_0^t\text{e}^{-c_2(t-\tau)}\|F(\tau)\|_{W^{1,6}}^2
\mathrm{d}\tau \\
&\le \text{e}^{-c_2t}(\rho(0)-1)^2+C\int_0^\frac{t}{2}\text{e}^{-c_2(t-\tau)}
\|F(\tau)\|_{W^{1,6}}^2\mathrm{d}\tau
+C\int_\frac{t}{2}^t\text{e}^{-c_2(t-\tau)}\|F(\tau)\|_{W^{1,6}}^2\mathrm{d}\tau \\
&\le \text{e}^{-c_2t}(\rho(0)-1)^2+C\text{e}^{-\frac{c_2t}{2}}\int_0^\frac{t}{2}\|F(\tau)\|_{W^{1,6}}^2\mathrm{d}\tau
+C\int_\frac{t}{2}^t\|F(\tau)\|_{W^{1,6}}^2\mathrm{d}\tau.
\end{align*}
Letting $t\rightarrow\infty$ and using \eqref{3.3}, we obtain the desired \eqref{3.6}.
\end{proof}

Next, we will prove the following key {\it a priori} estimates on $\nabla \varrho$, which plays a crucial role in our analysis.
\begin{Lemma}\label{lemma3.2}
Under the assumptions of Theorem \ref{thm1.1}, there exists a positive constant $C$ depending on the initial data $(\rho_0,{\bf u}_0)$, $c_1$, and $M$ such that
\begin{align}\label{3.12}
\|\varrho\|_{L^\infty((0,+\infty);W^{1,6})\cap L^2((0,+\infty);W^{1,6})}+\|\nabla{\bf u}\|_{L^2((0,+\infty);L^\infty)}\le C,
\end{align}
and
\begin{align}\label{3.13}
\lim\limits_{t\rightarrow\infty}\|\nabla\varrho(\cdot,t)\|_{L^6}=0.
\end{align}
\end{Lemma}
\begin{proof}
By \eqref{1.1}$_1$ and the definition of  $F$ in \eqref{1.4}, it is not difficult to derive that
\begin{align}\label{3.14}
\partial_t\nabla\varrho+\frac{\gamma}{2\mu+\lambda}\rho^\gamma\nabla \varrho=-\frac{\rho^\gamma-1}{2\mu+\lambda}\nabla\varrho-\frac{1}{2\mu+\lambda}\nabla\varrho F-\frac{1}{2\mu+\lambda}\rho \nabla F-\nabla{\bf u}\cdot\nabla\varrho-{\bf u}\cdot\nabla\nabla\varrho.
\end{align}
Multiplying \eqref{3.14} by $6|\nabla\varrho|^4\nabla\varrho$ and integrating the resulting equation on $\mathbb R^3$, one can deduce that
\begin{align}\label{3.15}
&\partial_t\|\nabla\varrho\|_{L^6}^6+\frac{6\gamma}{2\mu+\lambda}\int \rho^\gamma|\nabla\varrho|^6\mathrm{d}{\bf x}\nonumber
\\ &=-\int \frac{6(\rho^\gamma-1)}{2\mu+\lambda}|\nabla\varrho|^6\mathrm{d}{\bf x}-\int\frac{6}{2\mu+\lambda} F|\nabla\varrho|^6\mathrm{d}{\bf x}-\int \frac{6}{2\mu+\lambda}\rho \nabla F|\nabla\varrho|^4\nabla\varrho\mathrm{d}{\bf x}\nonumber
\\&\ \ \ -6\int \nabla{\bf u}\cdot\nabla\varrho|\nabla\varrho|^4\nabla\varrho\mathrm{d}{\bf x}-\int {\bf u}\cdot\nabla|\nabla\varrho|^6\mathrm{d}{\bf x}\nonumber
\\ &\le C\|\varrho\|_{L^\infty}\|\nabla\varrho\|_{L^6}^6+C\|F\|_{L^\infty}\|\nabla\varrho\|_{L^6}^6+C\|\nabla F\|_{L^6}\|\nabla\varrho\|_{L^6}^5\nonumber
\\&\ \ \ +C\|\nabla{\bf u}\|_{L^\infty}\|\nabla\varrho\|_{L^6}^6+C\|\divv{\bf u}\|_{L^\infty}\|\nabla\varrho\|_{L^6}^6\nonumber
\\ &\le C\|\varrho\|_{L^\infty}\|\nabla\varrho\|_{L^6}^6+C\| F\|_{W^{1,6}}\|\nabla\varrho\|_{L^6}^5
+C\|\nabla{\bf u}\|_{L^\infty}\|\nabla\varrho\|_{L^6}^6.
\end{align}
By \eqref{1.1}$_2$, we see that
\begin{equation}\nonumber \mu\Delta{\bf u}+(\mu+\lambda)\nabla\divv{\bf u}=\rho\dot{\bf u}+\nabla P(\rho).
\end{equation}
Thus, for any $1<p<\infty$, it follows from the standard $L^p$-estimate for the elliptic system that
\begin{align}\label{3.16}
\|\nabla^2{\bf u}\|_{L^p}\le C(\|\rho\dot{\bf u}\|_{L^p}+\|\nabla P\|_{L^p}),
\end{align}
which together with \eqref{2.9} yields that
\begin{align}\label{3.17}
\|\nabla{\bf u}\|_{L^\infty}&\le C\big(\|\divv{\bf u}\|_{L^\infty}+\|{\bf w}\|_{L^\infty}\big)\ln (\text{e}+\|\nabla^2{\bf u}\|_{L^6})+C\|\nabla{\bf u}\|_{L^2}+C\nonumber
\\&\le C\big(\|\divv{\bf u}\|_{L^\infty}+\|{\bf w}\|_{L^\infty}\big)\ln (\text{e}+\|\dot{\bf u}\|_{L^6}+\|\nabla\varrho\|_{L^6})+C\|\nabla{\bf u}\|_{L^2}+C\nonumber
\\&\le C\big(\|\divv{\bf u}\|_{L^\infty}+\|{\bf w}\|_{L^\infty}\big)\ln (\text{e}+\|\nabla\dot{\bf u}\|_{L^2})\nonumber
\\&\ \ \ + C\big(\|\divv{\bf u}\|_{L^\infty}+\|{\bf w}\|_{L^\infty}\big)\ln (\text{e}+\|\nabla\varrho\|_{L^6})+C\|\nabla{\bf u}\|_{L^2}+C.
\end{align}
Setting
\begin{align*}
f(t)\triangleq \text{e}+\|\nabla\varrho\|_{L^6}^6,
\ g(t)\triangleq 1+\big(\|\divv{\bf u}\|_{L^\infty}+\|{\bf w}\|_{L^\infty}\big)\ln (\text{e}+\|\nabla\dot{\bf u}\|_{L^2})+\|\nabla{\bf u}\|_{L^2}+\|\varrho\|_{L^\infty}+\| F\|_{W^{1,6}},
\end{align*}
then one obtains from \eqref{3.15} and \eqref{3.17} that
\begin{align*}
f'(t)\le Cg(t)f(t)+Cg(t)f(t)\ln f(t),
\end{align*}
which leads to
\begin{align}\label{3.18}
(\ln f(t))'\le Cg(t)+Cg(t)\ln f(t).
\end{align}
It follows from \eqref{1.4}, \eqref{1.8}, and \eqref{3.3} that
\begin{align*}
\int_0^Tg(t)\mathrm{d}t\le &C\int_0^T\big(1+\|\divv {\bf u}\|^2_{L^\infty}+\| {\bf w}\|^2_{L^\infty}+\|\nabla\dot{\bf u}\|_{L^2}^2+\|\nabla{\bf u}\|_{L^2}+\|\varrho\|_{L^\infty}+\| F\|_{W^{1,6}}\big)\mathrm{d}t\nonumber
\\ \le &C(T)+C\int_0^T\big(\|F\|_{L^\infty}^2+\|{\bf w}\|_{L^\infty}^2+\|\sigma\|_{L^\infty}^2\big)\mathrm{d}t\nonumber
\\ \le &C(T)+C\int_0^T\big(\|F\|_{W^{1,6}}^2+\|{\bf w}\|_{W^{1,6}}^2\big)\mathrm{d}t\nonumber
\\ \le &C(T),
\end{align*}
which together with \eqref{3.18} and Gronwall inequality shows that
\begin{align*}
\sup\limits_{0\le t\le T}f(t)\le C(T).
\end{align*}
Consequently, one has that
\begin{align}\label{3.19}
\sup\limits_{0\le t\le T}\|\nabla\rho\|_{L^6}\le C(T).
\end{align}

Combining \eqref{3.3} and \eqref{3.6}, we arrive at
\begin{align*}
\lim\limits_{t\rightarrow\infty}\big(\|\rho-1\|_{L^p}+\|\rho-1\|_{L^\infty}\big)=0,\  \ \text{for any}\ p>2.
\end{align*}
Therefore, for any $\delta>0$, there exists a constant $T_4>T_1$ such that, for any $t\ge T_4$ and $p>2$,
\begin{align}\label{3.20}
\|(\rho-1)(t)\|_{L^p}+\|(\rho-1)(t)\|_{L^\infty}\le \delta^2.
\end{align}
Now we claim that there exists a positive constant $C$ independent of $t$ such that, for any $t\ge 0$,
\begin{align}\label{3.21}
\|\nabla \varrho(t)\|_{L^6}\le C.
\end{align}
Indeed, if \eqref{3.21} is false, there exists a time $T_5>T_4$ such that  $\|\nabla \varrho(T_5)\|_{L^6}=\frac{1}{\delta^\kappa}$ for a positive constant $\kappa$. We choose a minimal value of $T_7>T_5$ such
that $\|\nabla \varrho(T_7)\|_{L^6}=\frac{2}{\delta^\kappa}$, then
choose a maximal value of $T_6<T_7$ such that $\|\nabla \varrho(T_6)\|_{L^6}=\frac{1}{\delta^\kappa}$.
Thus, one has that
\begin{align}\label{wz3}
\|\nabla \varrho(t)\|_{L^6}\in\Big[\frac{1}{\delta^\kappa},\frac{2}{\delta^\kappa}\Big]\ \ \text{for}\ \ t\in[T_6,T_7].
\end{align}
Therefore, for any $t\in[T_6,T_7]$, one obtains from Bernstein inequality and the interpolation inequality that
\begin{align}\label{wz7}
\|\nabla^2(-\Delta)^{-1}\sigma\|_{L^\infty}\le &\Big{\|}\nabla^2(-\Delta)^{-1}\sum\limits_{j\le N}\Delta_j\sigma\Big{\|}_{L^\infty}+
 \Big{\|}\nabla^2(-\Delta)^{-1}\sum\limits_{j\ge N}\Delta_j\sigma\Big{\|}_{L^\infty}\nonumber
 \\ \le &C2^\frac{3N}{p}\|\nabla^2(-\Delta)^{-1}\sigma\|_{L^p}+C\sum\limits_{j\ge N}2^{-\frac{j}{2}}\|\nabla\sigma\|_{L^6}\nonumber
 \\ \le &C\big(2^\frac{3N}{p}\|\varrho\|_{L^p}+2^{-\frac{N}{2}}\|\nabla\varrho\|_{L^6}\big)
 \nonumber
  \\ \le &C\delta,
\end{align}
provided that $N\ge \log_2\frac{1}{\delta^{2(\kappa+1)}}$ and $p=3N$. This implies that
\begin{align}\label{3.22}
\|\nabla{\bf u}\|_{L^\infty}\le
& C\big(\|\nabla\mathcal P {\bf u}\|_{L^\infty}+\|\nabla\mathcal Q {\bf u}\|_{L^\infty}\big)\nonumber
\\ \le &C\big(\|\nabla\mathcal P {\bf u}\|_{W^{1,6}}+\|\nabla^2(-\Delta)^{-1}F\|_{L^\infty}
+\|\nabla^2(-\Delta)^{-1}\sigma\|_{L^\infty}\big)\nonumber
\\ \le &C\big(\| {\bf w}\|_{W^{1,6}}+\|F\|_{W^{1,6}}+\delta\big),
\end{align}
where we have used
\begin{align*}
\mathcal Q {\bf u}=-\frac{\nabla(-\Delta)^{-1} (F+\sigma)}{2\mu+\lambda}.
\end{align*}
Inserting \eqref{3.22} into \eqref{3.15}, we see that
\begin{align*}
&\partial_t\|\nabla\varrho\|_{L^6}^6+\frac{6\gamma}{2\mu+\lambda}\int \rho^\gamma|\nabla\varrho|^6\mathrm{d}{\bf x}\nonumber
\\ &\le C\delta\|\nabla\varrho\|_{L^6}^6+C\| F\|_{W^{1,6}}\|\nabla\varrho\|_{L^6}^5
+C\big(\| {\bf w}\|_{W^{1,6}}+\|F\|_{W^{1,6}}\big)\|\nabla\varrho\|_{L^6}^6.
\end{align*}
Thus, if $\delta$ is suitably small, we arrive at
\begin{align*}
\partial_t\|\nabla\varrho\|_{L^6}^6+\frac{3\gamma}{2\mu+\lambda}\int \rho^\gamma|\nabla\varrho|^6\mathrm{d}{\bf x}
\le C\| F\|_{W^{1,6}}\|\nabla\varrho\|_{L^6}^5
+C\big(\| {\bf w}\|_{W^{1,6}}+\|F\|_{W^{1,6}}\big)\|\nabla\varrho\|_{L^6}^6.
\end{align*}
Integrating the above inequality with respect to $t$ from $T_6$ to $T_7$ and using \eqref{3.7}, one finds that
\begin{align*}
&\frac{1}{\delta^{6\kappa}}+\frac{3\gamma}{(2\mu+\lambda)2^\gamma}\frac{1}{\delta^{6\kappa}}(T_7-T_6)\nonumber
\\&\le C\left(\int_{T_6}^{T_7}(\| {\bf w}\|_{W^{1,6}}^2+\|F\|_{W^{1,6}}^2)\mathrm{d}t\right)^\frac{1}{2}(T_7-T_6)^\frac12
\sup\limits_{T_6\le t\le T_7}\big(\|\nabla\varrho\|_{L^6}^5+\|\nabla\varrho\|_{L^6}^6\big)\nonumber
\\&\le C\frac{\delta^\frac{1}{2}}{\delta^{6\kappa}}(T_7-T_6)^\frac12,
\end{align*}
which is impossible when $\delta$ is small enough and \eqref{3.21} follows
immediately.

For any $t\ge T_4$, substituting \eqref{3.20} and \eqref{3.21} into \eqref{wz7} gives that
\begin{align*}
\|\nabla^2(-\Delta)^{-1}\sigma\|_{L^\infty} \le C\left(2^\frac{3N}{p}\|\varrho\|_{L^p}+2^{-\frac{N}{2}}\|\nabla\varrho\|_{L^6}\right)
\le C_12^\frac{3N}{p}\delta^2+C_22^{-\frac{N}{2}}
\le \delta,
\end{align*}
provided that $\delta<\frac{1}{2C_1}$, $N\ge 2\log_2\frac{C_2}{\delta}$, and $p=3N$. Plugging the above estimate into \eqref{3.22} yields that
\begin{align}\label{3.23}
\|\nabla{\bf u}\|_{L^\infty}
\le C\big(\| {\bf w}\|_{W^{1,6}}+\|F\|_{W^{1,6}}+\delta\big).
\end{align}
Inserting \eqref{3.23} into \eqref{3.15}, we derive that
\begin{align*}
&\partial_t\|\nabla\varrho\|_{L^6}^6+\frac{6\gamma}{2\mu+\lambda}\int \rho^\gamma|\nabla\varrho|^6\mathrm{d}{\bf x}\nonumber
\\ &\le C\delta\|\nabla\varrho\|_{L^6}^6+C\| F\|_{W^{1,6}}\|\nabla\varrho\|_{L^6}^5
+C\|{\bf w}\|_{W^{1,6}}\|\nabla\varrho\|_{L^6}^6+C\| F\|_{W^{1,6}}\|\nabla\varrho\|_{L^6}^6,
\end{align*}
which together with \eqref{3.21} implies that there exists a positive constant $c_3$ such that
\begin{align}
\partial_t\|\nabla\varrho\|_{L^6}+c_3\|\nabla\varrho\|_{L^6}
&\le C\| F\|_{W^{1,6}}
+C\|{\bf w}\|_{W^{1,6}}\|\nabla\varrho\|_{L^6}+C\| F\|_{W^{1,6}}\|\nabla\varrho\|_{L^6}\nonumber
\\ &\le C\| F\|_{W^{1,6}}
+C\|{\bf w}\|_{W^{1,6}}.\label{3.25}
\end{align}
By Gronwall inequality, we have
\begin{align*}
\|\nabla\varrho(t)\|_{L^6}\le & \text{e}^{-c_3(t-T_4)}\|\nabla\varrho(T_4)\|_{L^6}
+C\int_{T_4}^t\text{e}^{-c_3(t-\tau)}
\big(\| F\|_{W^{1,6}}+\|{\bf w}\|_{W^{1,6}}\big)\mathrm{d}\tau\nonumber
\\ \le &\text{e}^{-c_3(t-T_4)}\|\nabla\varrho(T_4)\|_{L^6}
+C\int_{T_4}^{\frac{t+T_4}{2}}\text{e}^{-c_3(t-\tau)}
\big(\| F\|_{W^{1,6}}+\|{\bf w}\|_{W^{1,6}}\big)\mathrm{d}\tau\nonumber
\\&+C\int^t_{\frac{t+T_4}{2}}\text{e}^{-c_3(t-\tau)}(\| F\|_{W^{1,6}}
+\|{\bf w}\|_{W^{1,6}})\mathrm{d}\tau\nonumber
\\ \le &\text{e}^{-c_3(t-T_4)}\|\nabla\varrho(T_4)\|_{L^6}
+C\frac{t-T_4}{2}\text{e}^{-\frac{c_3(t+T_4)}{2}}\left(\int_{T_4}^{\frac{t+T_4}{2}}
\big(\| F\|_{W^{1,6}}^2+\|{\bf w}\|_{W^{1,6}}^2\big)\mathrm{d}\tau\right)^\frac{1}{2}\nonumber
\\&+C\left(\int^t_{\frac{t+T_4}{2}}\text{e}^{-2c_3(t-\tau)}
\mathrm{d}\tau\right)^\frac{1}{2}\left(\int^t_{\frac{t+T_4}{2}}
\big(\|F\|_{W^{1,6}}^2+\|{\bf w}\|^2_{W^{1,6}}\big)
\mathrm{d}\tau\right)^\frac{1}{2}\nonumber
\\ \le &\text{e}^{-c_3(t-T_4)}\|\nabla\varrho(T_4)\|_{L^6}+C\text{e}^{-\frac{c_3(t+T_4)}{4}}
\delta^\frac{1}{2}
+C\left(\int^t_{\frac{t+T_4}{2}}\big(\| F\|_{W^{1,6}}^2
+\|{\bf w}\|^2_{W^{1,6}}\big)\mathrm{d}\tau\right)^\frac{1}{2},
\end{align*}
which leads to \eqref{3.13}.

Finally, multiplying \eqref{3.25} by $2\|\nabla\varrho\|_{L^6}$ and using Cauchy--Schwarz inequality, we see that
\begin{align}\label{3.26}
\partial_t&\|\nabla\varrho\|_{L^6}^2+2c_3\|\nabla\varrho\|_{L^6}^2
\le C\| F\|_{W^{1,6}}^2
+C\|{\bf w}\|_{W^{1,6}}^2.
\end{align}
Integrating \eqref{3.26} with respect to $t$ from $0$ to $\infty$ and using \eqref{3.3}, one has that
\begin{align}\int_0^\infty\|\nabla\varrho\|_{L^6}^2\mathrm{d}t\le C\int_0^\infty\left(\|F\|_{W^{1,6}}^2+\|{\bf w}\|_{W^{1,6}}^2\right)\mathrm{d}t\le C,\nonumber
\end{align}
and
\begin{align}\nonumber\int_0^\infty\|\nabla{\bf u}\|_{L^\infty}^2\mathrm{d}t\le C\int_0^\infty\|\nabla{\bf u}\|_{W^{1,6}}^2\mathrm{d}t
 \le C\int_0^\infty\left(\|F\|_{W^{1,6}}^2+\|\sigma\|_{W^{1,6}}^2+\|{\bf w}\|_{W^{1,6}}^2\right)\mathrm{d}t
 \le C,\label{3.28}
\end{align}
these together with \eqref{3.21} imply \eqref{3.12}.
\end{proof}

Next, we will prove the following uniform-in-time bounds on the solution.
\begin{Lemma}\label{lemma3.3}
Under the assumptions of Theorem \ref{thm1.1}, it holds that
\begin{equation}\sup\limits_{0\le t<\infty}\left(\|\varrho(t)\|_{H^3}^2+\|{\bf u}(t)\|_{H^3}^2\right)+\int_0^\infty\left(\|\nabla\varrho(t)\|_{H^2}^2+\|\nabla{\bf u}(t)\|_{H^3}^2\right)\mathrm{d}t\le C.\label{3.29}
\end{equation}
\end{Lemma}
\begin{proof}
Multiplying \eqref{3.14} by $2\nabla\varrho$ and integrating the resulting equation over $\mathbb R^3$, one derives that
\begin{align*}
&\partial_t\|\nabla\varrho\|_{L^2}^2+\frac{2\gamma}{2\mu+\lambda}\int \rho^\gamma|\nabla\varrho|^2\mathrm{d}{\bf x}\nonumber
\\ &=-\int \frac{2(\rho^\gamma-1)}{2\mu+\lambda}|\nabla\varrho|^2\mathrm{d}{\bf x}-\int\frac{2}{2\mu+\lambda} F|\nabla\varrho|^2\mathrm{d}{\bf x}-\int \frac{2}{2\mu+\lambda}\rho \nabla F\cdot\nabla\varrho\mathrm{d}{\bf x}\nonumber
\\&\ \ \ -2\int \nabla{\bf u}\cdot\nabla\varrho\cdot\nabla\varrho\mathrm{d}{\bf x}
-\int {\bf u}\cdot\nabla|\nabla\varrho|^2\mathrm{d}{\bf x}\nonumber
\\ &\le C\|\varrho\|_{L^\infty}\|\nabla\varrho\|_{L^2}^2
+C\|F\|_{L^\infty}\|\nabla\varrho\|_{L^2}^2+C\|\nabla F\|_{L^2}\|\nabla\varrho\|_{L^2}\nonumber
\\&\ \ \ +C\|\nabla{\bf u}\|_{L^\infty}\|\nabla\varrho\|_{L^2}^2+C\|\divv{\bf u}\|_{L^\infty}\|\nabla\varrho\|_{L^2}^2\nonumber
\\ &\le C\|\varrho\|_{L^\infty}\|\nabla\varrho\|_{L^2}^2+C\| F\|_{W^{1,6}}\|\nabla\varrho\|_{L^2}^2+C\|\nabla F\|_{L^2}\|\nabla\varrho\|_{L^2}
+C\|\nabla{\bf u}\|_{L^\infty}\|\nabla\varrho\|_{L^2}^2,
\end{align*}
which combined with \eqref{3.5} and Cauchy--Schwarz inequality implies that
\begin{align}\label{3.30}
&\partial_t\|\nabla\varrho\|_{L^2}^2+\frac{\gamma}{2\mu+\lambda}\int \rho^\gamma|\nabla\varrho|^2\mathrm{d}{\bf x}\nonumber
\\ &\le C\|\varrho\|_{L^\infty}\|\nabla\varrho\|_{L^2}^2+C\| F\|_{W^{1,6}}^2\|\nabla\varrho\|_{L^2}^2+C\|\nabla F\|_{L^2}^2
+C\|\nabla{\bf u}\|_{L^\infty}^2\|\nabla\varrho\|_{L^2}^2.
\end{align}
This along with Gronwall inequality, \eqref{3.3}, and \eqref{3.12} leads to
\begin{align}\label{3.31}
\|\nabla\varrho(T)\|_{L^2}\le C(T).
\end{align}

For any $\delta>0$, it infers from \eqref{3.6} and \eqref{3.13} that there exists a constant $T_8>T_1$ such that
\begin{align}\label{3.32}
\|\varrho(\cdot,t)\|_{L^\infty}+\|\nabla\varrho(\cdot,t)
\|_{L^6}\le \delta,\ \  \text{for any}\ t\ge T_8.
\end{align}
Therefore, by \eqref{3.30}, it holds that, for any $t\ge T_8$,
\begin{align}\label{3.33}
\partial_t\|\nabla\varrho\|_{L^2}^2+\frac{\gamma}{2(2\mu+\lambda)}\int\rho^\gamma |\nabla\varrho|^2\mathrm{d}{\bf x}
\le C\| F\|_{W^{1,6}}^2\|\nabla\varrho\|_{L^2}^2+C\|\nabla F\|_{L^2}^2
+C\|\nabla{\bf u}\|_{L^\infty}^2\|\nabla\varrho\|_{L^2}^2,
\end{align}
which together with Gronwall inequality, \eqref{3.3}, \eqref{3.12}, and \eqref{3.31} yields that
\begin{align*}
\|\nabla\varrho(t)\|_{L^2}\le C,\ \  \text{for any}\ t\ge 0.
\end{align*}
This combined with \eqref{3.3}, \eqref{3.5} and \eqref{3.33} implies that
\begin{align}\sup\limits_{0\le t<\infty}\|\nabla\varrho(t)\|_{L^2}^2+\int_0^\infty\|\nabla\varrho(t)\|_{L^2}^2\mathrm{d}t\le C,
\label{3.34}
\end{align}
\begin{equation}
\int_0^\infty\|\nabla^2{\bf u}\|_{L^2}^2\mathrm{d}t\le C\int_0^\infty\left(\|\nabla F\|_{L^2}^2+\|\nabla {\bf w}\|_{L^2}^2+\|\nabla \varrho\|_{L^2}^2\right)\mathrm{d}t
\le C,\label{3.35}
\end{equation}
and
\begin{align}\label{3.36}
\sup\limits_{0\le t<\infty}\|\nabla^2{\bf u}(t)\|_{L^2}^2\le C\sup\limits_{0\le t<\infty}\left(\|\sqrt{\rho}\dot{\bf u} \|_{L^2}^2+\|\nabla \varrho\|_{L^2}^2\right)
\le C.
\end{align}

Applying the operator $\nabla$ to \eqref{3.14}, one gets that
\begin{align}\label{3.37}
\partial_t\nabla^2\varrho+\frac{\gamma}{2\mu+\lambda}\rho^\gamma\nabla^2 \varrho =&-\frac{2\gamma^2}{2\mu+\lambda}\rho^{\gamma-1}\nabla \varrho\nabla \varrho-\frac{\rho^\gamma-1}{2\mu+\lambda}\nabla^2\varrho-\frac{1}{2\mu+\lambda}\nabla^2\varrho F-\frac{2}{2\mu+\lambda}\nabla\varrho \nabla F
\notag \\ &-\frac{1}{2\mu+\lambda}\rho \nabla^2 F-\nabla^2{\bf u}\cdot\nabla\varrho-2\nabla{\bf u}\cdot\nabla\nabla\varrho-{\bf u}\cdot\nabla\nabla^2\varrho.
\end{align}
Multiplying \eqref{3.37} by $2\nabla^2\varrho$ and integrating the resulting equation over $\mathbb R^3$, we then infer from \eqref{3.5} and Cauchy--Schwarz inequality that
\begin{align}\label{3.38}
&\partial_t\|\nabla^2\varrho\|_{L^2}^2+\frac{\gamma}{2\mu+\lambda}\int \rho^\gamma|\nabla^2\varrho|^2\mathrm{d}{\bf x}\nonumber
\\& \le C\|\varrho\|_{L^\infty}\|\nabla^2\varrho\|_{L^2}^2+C\|\nabla\varrho\|_{L^3}^2\|\nabla\varrho\|_{L^6}^2
+C\|F\|_{L^\infty}^2\|\nabla^2\varrho\|_{L^2}^2
+C\|\nabla F\|_{L^6}^2\|\nabla\varrho\|_{L^3}^2
\notag\\ &\quad +C\|\nabla^2F\|_{L^2}^2+C\|\nabla^2 {\bf u}\|_{L^6}^2\|\nabla\varrho\|_{L^3}^2+C\|\nabla {\bf u}\|_{L^\infty}^2\|\nabla^2\varrho\|_{L^2}^2.
\end{align}
Noting that
\begin{align}\label{3.39}
&\|\nabla^2 {\bf u}\|_{L^6}^2
\le C\big(\|\rho\dot {\bf u}\|_{L^6}^2+\|\nabla\sigma\|_{L^6}^2\big)
\le C\big(\|\nabla\dot {\bf u}\|_{L^2}^2+\|\nabla\varrho\|_{L^6}^2\big),\\ \label{3.40}
&\|\nabla^2 F\|_{L^2}^2\le C\|\nabla(\rho\dot {\bf u})\|_{L^2}^2\le C\|\nabla\dot {\bf u}\|_{L^2}^2+C\|\nabla\varrho\|_{L^3}^2\|\dot {\bf u}\|_{L^6}^2\le C\|\nabla\dot {\bf u}\|_{L^2}^2,
\end{align}
one thus derives from \eqref{3.38} that
\begin{align}\label{3.41}
\partial_t\|&\nabla^2\varrho\|_{L^2}^2+\frac{\gamma}{2\mu+\lambda}\int \rho^\gamma|\nabla^2\varrho|^2\mathrm{d}{\bf x}\nonumber
\\ \le &C\|\varrho\|_{L^\infty}\|\nabla^2\varrho\|_{L^2}^2+C\|\nabla\varrho\|_{L^3}^2\|\nabla\varrho\|_{L^6}^2
+C\|F\|_{W^{1,6}}^2\|\nabla^2\varrho\|_{L^2}^2
+C\|\nabla F\|_{L^6}^2\|\nabla\varrho\|_{L^3}^2
\notag\\ &+C\|\nabla\dot{\bf u}\|_{L^2}^2+C\|\nabla\dot {\bf u}\|_{L^2}^2\|\nabla\varrho\|_{L^3}^2+C\|\nabla {\bf u}\|_{L^\infty}^2\|\nabla^2\varrho\|_{L^2}^2.
\end{align}
This together with Gronwall inequality, \eqref{3.3}, \eqref{3.5}, \eqref{3.12}, and \eqref{3.34} shows that
\begin{align}\label{3.42}
\|\nabla^2\varrho(T)\|_{L^2}\le C(T).
\end{align}

By \eqref{3.12}, \eqref{3.32}, \eqref{3.34}, and \eqref{3.41}, it holds that, for any $t\ge T_8$,
\begin{align}\label{3.43}
\partial_t\|\nabla^2\varrho\|_{L^2}^2+\frac{\gamma}{2(2\mu+\lambda)}\int |\nabla^2\varrho|^2\mathrm{d}{\bf x}
& \le C\|\nabla\varrho\|_{L^3}^2\|\nabla\varrho\|_{L^6}^2
+C\|F\|_{W^{1,6}}^2\|\nabla^2\varrho\|_{L^2}^2
+C\|\nabla F\|_{L^6}^2
\notag\\ & \quad +C\|\nabla^2F\|_{L^2}^2+C\|\nabla\dot {\bf u}\|_{L^2}^2+C\|\nabla {\bf u}\|_{L^\infty}^2\|\nabla^2\varrho\|_{L^2}^2.
\end{align}
This along with Gronwall inequality, \eqref{3.3}, \eqref{3.12}, \eqref{3.34}, and \eqref{3.42} indicates that
\begin{align*}
\|\nabla^2\varrho(t)\|_{L^2}\le C,\ \  \text{for any}\ t\ge 0,
\end{align*}
which combined with \eqref{3.43} implies that
\begin{align}\sup\limits_{0\le t<\infty}\|\nabla^2\varrho(t)\|_{L^2}^2+\int_0^\infty\|\nabla^2\varrho(t)\|_{L^2}^2\mathrm{d}t\le C,
\label{3.44}
\end{align}
and
\begin{align}\int_0^\infty\|\nabla^3{\bf u}\|_{L^2}^2\mathrm{d}t\le &C\int_0^\infty\left(\|\nabla (\rho\dot{\bf u})\|_{L^2}^2+\|\nabla^2 \sigma\|_{L^2}^2\right)\mathrm{d}t\notag
\\ \le &C\int_0^\infty\left(\|\nabla \dot{\bf u}\|_{L^2}^2+\| \nabla\varrho\|_{L^3}^2\| \dot{\bf u}\|_{L^6}^2+\|\nabla^2 \varrho\|_{L^2}^2+\|\nabla \varrho\|_{L^4}^4\right)\mathrm{d}t
\notag
\\ \le &C.\label{3.45}
\end{align}

The system \eqref{1.1} can be equivalently reformulated in the linearized
form as follows
\begin{align}\label{3.46}
\begin{cases}
\varrho_{t}+\divv\mathbf{u}=-\varrho\divv\mathbf{u}-\mathbf{u}\cdot\nabla\varrho,\\
\mathbf{u}_{t}
-\mu\Delta\mathbf{u}-(\mu+\lambda)\nabla\divv\mathbf{u}+\gamma\nabla \varrho=-\mathbf{u}\cdot\nabla\mathbf{u}-\frac{\mu\varrho}{\rho}\Delta\mathbf{u}
-\frac{(\mu+\lambda)\varrho}{\rho}\nabla\divv\mathbf{u}-\gamma(\rho^{\gamma-1}-1)\nabla\varrho.
\end{cases}
\end{align}
Applying $\nabla^3$ to \eqref{3.46}$_1$ and \eqref{3.46}$_2$, then multiplying them by $\gamma\nabla^3\varrho$ and $\nabla^3{\bf u}$, respectively, and integrating the resultant equations over $\mathbb R^3$ and summing up, it is not difficult to deduce that
\begin{align}\label{3.47}
\frac{1}{2}\frac{\mathrm{d}}{\mathrm{d}t}&\left(\gamma\|\nabla^3\varrho\|_{L^2}^2
+\|\nabla^3{\bf u}\|_{L^2}^2\right)
+\mu\|\nabla^4{\bf u}\|_{L^2}^2+(\mu+\lambda)\|\nabla^3\divv{\bf u}\|_{L^2}^2\notag
\\ \le & C\|\nabla{\bf u}\|_{H^2}\|\nabla^3\varrho\|_{L^2}^2+C\|\nabla\varrho\|_{H^1}\|\nabla^4{\bf u}\|_{L^2}\|\nabla^3\varrho\|_{L^2}+C\|\varrho\|_{L^\infty}\|\nabla^4{\bf u}\|_{L^2}\|\nabla^3\varrho\|_{L^2}\notag
\\ & +C\|\nabla{\bf u}\|_{H^2}^2\|\nabla^4{\bf u}\|_{L^2}+C\|\varrho\|_{L^\infty}\|\nabla^4{\bf u}\|_{L^2}^2+C\|\nabla^2\varrho\|_{L^2}\|\nabla^3{\bf u}\|_{L^2}^\frac{1}{2}\|\nabla^4{\bf u}\|_{L^2}^\frac{3}{2}\notag
\\ & +C\|\nabla^3\varrho\|_{L^2}\|\nabla^2{\bf u}\|_{H^1}\|\nabla^4{\bf u}\|_{L^2}+C\|\nabla\varrho\|_{H^1}^2\|\nabla^2{\bf u}\|_{H^1}\|\nabla^4{\bf u}\|_{L^2}\notag
\\ & +C\|\nabla\varrho\|_{H^1}\|\nabla^3\varrho\|_{L^2}^2
+C\|\nabla\varrho\|_{H^1}^3\|\nabla^3\varrho\|_{L^2}.
\end{align}
Applying $\nabla^2$ to \eqref{3.46}$_2$ and multiplying the resulting equations by $\nabla^3\varrho$, and then integrating over $\mathbb R^3$, one has that
\begin{align*}
\gamma\|\nabla^3\varrho\|_{L^2}^2= &-\int \nabla^2{\bf u}_t\nabla^3\varrho\mathrm{d}{\bf x}+\int \nabla^2\left(\frac{\mu\Delta\mathbf{u}+(\mu+\lambda)\divv\mathbf{u}}{\rho}-{\bf u}\cdot\nabla{\bf u}
\right)\nabla^3\varrho\mathrm{d}{\bf x}\notag
\\ &-\gamma\int\nabla^2\left((\rho^{\gamma-1}-1)\varrho\right)\nabla^3\varrho\mathrm{d}{\bf x}\notag
\\ \le &\frac{\mathrm{d}}{\mathrm{d}t}\int \nabla^2\divv{\bf u}\nabla^2\varrho\mathrm{d}{\bf x}+C\|\nabla^3{\bf u}\|_{L^2}^2
+C\|\nabla\varrho\|_{H^1}\|\nabla{\bf u}\|_{H^1}\|\nabla^4{\bf u}\|_{L^2}+C\|\nabla^3\varrho\|_{L^2}\|\nabla^4{\bf u}\|_{L^2}\notag
\\ &+C\|\nabla\varrho\|_{H^1}\|\nabla^3\varrho\|_{L^2}\|\nabla^4{\bf u}\|_{L^2}+C\|\nabla^3\varrho\|_{L^2}\|\nabla{\bf u}\|_{H^2}^2+C\|\nabla^3\varrho\|_{L^2}\|\nabla\varrho\|_{H^1}^2,
\end{align*}
which implies that
\begin{align}\label{3.48}
-\frac{\mathrm{d}}{\mathrm{d}t}\int \nabla^2\divv{\bf u}\nabla^2\varrho\mathrm{d}{\bf x}+\frac{\gamma}{2}\|\nabla^3\varrho\|_{L^2}^2
\le &C\|\nabla^3{\bf u}\|_{L^2}^2
+C\|\nabla\varrho\|_{H^1}^2\|\nabla{\bf u}\|_{H^1}^2\notag
+C\|\nabla^4{\bf u}\|_{L^2}^2\notag
\\ &+C\|\nabla\varrho\|_{H^1}^2\|\nabla^3\varrho\|_{L^2}^2+C\|\nabla{\bf u}\|_{H^2}^4+C\|\nabla\varrho\|_{H^1}^4.
\end{align}
Let $D>0$ be a large fixed constant. Summing up $D\times\eqref{3.47}+\eqref{3.48}+\eqref{3.41}$ and applying Cauchy--Schwarz inequality,  then there exists an energy functional $H(\|\nabla^2\varrho\|_{H^1},\|\nabla^3{\bf u}\|_{L^2})$ which
is equivalent to $\|\nabla^2\varrho\|_{H^1}^2+\|\nabla^3{\bf u}\|_{L^2}^2$ such that
\begin{align}
\frac{\mathrm{d}}{\mathrm{d}t}H(t)+C_1\left(\|\nabla^2\varrho\|_{H^1}^2+\|\nabla^4{\bf u}\|_{L^2}^2\right)
\le &C(\|\nabla{\bf u}\|^2_{H^2}+\|\nabla\varrho\|^2_{H^1})H(t)+C\|\nabla\varrho\|_{H^1}^2+\|\nabla{\bf u}\|_{H^1}^2
\notag\\ &+C\|F\|_{W^{1,6}}^2\|\nabla^2\varrho\|_{L^2}^2+C\|\nabla F\|_{L^6}^2+C\|\nabla\dot{\bf u}\|_{L^2}^2,\label{3.49}
\end{align}
where we have used \eqref{3.3}, \eqref{3.12}, \eqref{3.34}, \eqref{3.36}, and \eqref{3.44}. Applying Gronwall inequality to \eqref{3.49}, we see that
\begin{align*}
H(t)\le C,
\end{align*}
which combined with \eqref{3.49} leads to
\begin{align*}
\sup\limits_{0\le t<\infty}\left(\|\nabla^2\varrho(t)\|_{H^1}^2+\|\nabla^3{\bf u}(t)\|_{L^2}^2\right)+\int_0^\infty\left(\|\nabla^2\varrho(t)\|_{H^1}^2+\|\nabla^4{\bf u}(t)\|_{L^2}^2\right)\mathrm{d}t\le C.
\end{align*}
This along with \eqref{3.3}, \eqref{3.34}, and Sobolev inequality yields \eqref{3.29}.
\end{proof}

The following Lemma \ref{lemma3.4} is necessary for decay estimates on the higher-order
derivatives of $(\varrho,{\bf u})$.
\begin{Lemma}\label{lemma3.4}
Under the assumptions of Theorem \ref{thm1.1},  it holds that
\begin{align}\label{3.51}
\sup\limits_{0\le t<\infty}\Big(\|\varrho(t)\|_{\dot B^{-s}_{2,\infty}}
+\|{\bf u}(t)\|_{\dot B^{-s}_{2,\infty}}\Big)\le C,\ \text{for}\ s\in(0,1/2],
\\ \label{3.52}
\|\varrho(T)\|_{\dot B^{-s}_{2,\infty}}+\|{\bf u}(T)\|_{\dot B^{-s}_{2,\infty}}\le C(T),\ \text{for}\ s\in(1/2,3/2].
\end{align}
\end{Lemma}
\begin{proof}Applying $2^{-2sq}\Delta_q\eqref{3.46}_1$, $2^{-2sq}\Delta_q\eqref{3.46}_2$ by $\gamma2^{-2sq}\Delta_q \varrho$, $2^{-2sq}\Delta_q{\bf u}$, respectively, summing up,  integrating the resultant equation over
$\mathbb R^3$, and then using \eqref{2.7}, we obtain that
\begin{align}\label{3.53}
\frac{1}{2}
&\frac{\mathrm{d}}{\mathrm{d}t}\left(\gamma\norm{2^{-2sq}\Delta_q\varrho}_{L^2}^2+\norm{2^{-2sq}\Delta_q{\bf u}}_{L^2}^2\right)
+\mu\norm{2^{-2sq}\Delta_q(\nabla{\bf u})}_{L^2}^2+(\mu+\lambda)\norm{2^{-2sq}\Delta_q(\divv{\bf u})}_{L^2}^2\notag
\\ =&\gamma\int 2^{-2sq}\Delta_q\divv(\varrho{\bf u})2^{-2sq}\Delta_q\varrho\mathrm{d}{\bf x}+\int 2^{-2sq}\Delta_q\left(-\mathbf{u}\cdot\nabla\mathbf{u}-\gamma(\rho^{\gamma-1}-1)\nabla\varrho\right)2^{-2sq}\Delta_q{\bf u}\mathrm{d}{\bf x}\notag \\&+\int 2^{-2sq}\Delta_q\left(-\frac{\mu\varrho}{\rho}\Delta\mathbf{u}
-\frac{(\mu+\lambda)\varrho}{\rho}\nabla\divv\mathbf{u}\right)2^{-2sq}\Delta_q{\bf u}\mathrm{d}{\bf x}\notag
\\ \le &C\left\|\left(\divv (\varrho{\bf u}),-\mathbf{u}\cdot\nabla\mathbf{u}-\gamma(\rho^{\gamma-1}-1)\nabla\varrho,-\frac{\mu\varrho}{\rho}\Delta\mathbf{u}
-\frac{(\mu+\lambda)\varrho}{\rho}\nabla\divv\mathbf{u}\right)\right\|_{\dot B^{-s}_{2,1}}\notag
\\ &\times\left(\gamma\norm{2^{-2sq}\Delta_q\varrho}_{L^2}^2+\norm{2^{-2sq}\Delta_q{\bf u}}_{L^2}^2\right)\notag
\\ \le &C\left\|\left(\divv (\varrho{\bf u}),-\mathbf{u}\cdot\nabla\mathbf{u}-\gamma(\rho^{\gamma-1}-1)\nabla\varrho,-\frac{\mu\varrho}{\rho}\Delta\mathbf{u}
-\frac{(\mu+\lambda)\varrho}{\rho}\nabla\divv\mathbf{u}\right)\right\|_{L^\frac{6}{3+2s}}\notag
\\ &\times\left(\gamma\norm{2^{-2sq}\Delta_q\varrho}_{L^2}^2+\norm{2^{-2sq}\Delta_q{\bf u}}_{L^2}^2\right)\notag
\\ \le &C\|(\varrho,{\bf u})\|_{L^\frac{3}{s}}\|\nabla(\varrho,{\bf u})\|_{H^1}\left(\gamma\norm{2^{-2sq}\Delta_q\varrho}_{L^2}^2+\norm{2^{-2sq}\Delta_q{\bf u}}_{L^2}^2\right).
\end{align}
If $s\in(0,1/2]$, that is $\frac3s\ge 6$, then \eqref{3.51} follows from \eqref{3.53}, Sobolev inequality, and Gronwall inequality.
When $s\in(1/2,3/2]$, then $\frac3s\ge 2$, so \eqref{3.52} follows from \eqref{3.53}, Gronwall inequality, and \eqref{3.29}.
\end{proof}

\section{Proofs of Theorems \ref{thm1.1} and \ref{thm1.2}}\label{sec4}

With all the {\it a priori} estimates in Section \ref{sec3} at hand, we are ready to prove Theorems \ref{thm1.1} and \ref{thm1.2}.

{\bf Proof of Theorem \ref{thm1.1}.}
The desired \eqref{z1.8} and \eqref{1.9} follow from \eqref{3.5} and \eqref{3.29}, respectively.

To show \eqref{1.10}, in virtue of \eqref{3.29}, there exists a positive constant $T_9$ such that
\begin{align*}
\|\nabla\varrho(T_9)\|_{H^2}+\|\nabla{\bf u}(T_9)\|_{H^2}\le \delta,
\end{align*}
which together with \eqref{2.8} and \eqref{3.51} gives that
\begin{align*}
\|\varrho(T_9)\|_{L^2}+\|{\bf u}(T_9)\|_{L^2} \le C\|\nabla(\varrho,{\bf u})\|_{L^2}^\frac{s}{1+s}\|\nabla(\varrho,{\bf u})\|_{\dot B^{-s}_{2,\infty}}^\frac{1}{1+s}\le C\delta^\frac{s}{1+s},\ \text{if}\ s\in(0,1/2],
\\
\|\varrho(T_9)\|_{L^2}+\|{\bf u}(T_9)\|_{L^2} \le C\|\nabla(\varrho,{\bf u})\|_{L^2}^\frac{1}{3}\|\nabla(\varrho,{\bf u})\|_{\dot B^{-\frac{1}{2}}_{2,\infty}}^\frac{2}{3}\le C\delta^\frac{1}{3},\ \text{if}\ s\in(1/2,3/2].
\end{align*}
Therefore, we have
\begin{align}\label{3.54}
\|\varrho(T_9)\|_{L^2}+\|{\bf u}(T_9)\|_{L^2} \le \delta^\frac{1}{6},
\end{align}
provided that $\delta$ is properly small. Combining \eqref{3.54} with \eqref{3.51} and \eqref{3.52}, then using a family of scaled energy estimates with minimum derivative counts and interpolations among them developed by Guo and Wang \cite{GuoY}, we deduce \eqref{1.10}. The proof of Theorem \ref{thm1.1} is completed. \hfill$\Box$

{\bf Proof of Theorem \ref{thm1.2}.}
It follows from \eqref{1.1}$_1$ and the definition of $\sigma$ that
\begin{align}\label{4.1}
\sigma_t+{\bf u}\cdot\nabla \sigma+\gamma \sigma\divv {\bf u}=0.
\end{align}
Multiplying \eqref{4.1} by $6\sigma^5$ and integrating the resulting equality over $\mathbb R^3$, one has that
\begin{align*}
\frac{\mathrm{d}}{\mathrm{d}t}\norm{\sigma}_{L^6}^6=-(4\gamma-1)\int \sigma^6\divv {\bf u}\mathrm{d}{\bf x}-\gamma\int \sigma^5\divv {\bf u}\mathrm{d}{\bf x},
\end{align*}
which combined with \eqref{1.4}, \eqref{1.8}, \eqref{3.3}, and Young inequality yields that
\begin{align}\nonumber
\int_0^\infty\left|\frac{\mathrm{d}}{\mathrm{d}t}\norm{\sigma}_{L^6}^6\right|\mathrm{d}t\le C\int_0^\infty\big(\norm{\sigma}_{L^6}^6+\norm{\divv {\bf u}}_{L^2}^2\big)\mathrm{d}t
\le C,
\end{align}
As a result, we have
\begin{align*}
\lim\limits_{t\rightarrow\infty}\norm{\sigma}_{L^6}=0.
\end{align*}
This along with \eqref{3.1} implies that
\begin{align}\label{wz8}
\lim\limits_{t\rightarrow\infty}\|\rho-1\|_{L^6}=0,
\end{align}
which together with \eqref{1.7}, \eqref{1.8}, \eqref{3.3}, and the interpolation inequality indicates \eqref{1.14}.

On the other hand, thanks to \eqref{3.3}, one gets that
\begin{align}
\inf\limits_{{\bf x}\in \mathbb R^3}\rho({\bf x},t)
\le \inf\limits_{{\bf x}\in \mathbb R^3}\rho_0({\bf x})\text{e}^{\int_0^t\|\divv {\bf u}\|_{L^\infty}\mathrm{d}\tau}
\le \inf\limits_{{\bf x}\in \mathbb R^3}\rho_0({\bf x})\text{e}^{C\int_0^t(\|F\|_{L^\infty}+\|\sigma\|_{L^\infty})\mathrm{d}\tau}
\le \inf\limits_{{\bf x}\in \mathbb R^3}\rho_0({\bf x})\text{e}^{C(1+t)} = 0,\notag
\end{align}
as desired \eqref{1.15}. We complete the proof of Theorem \ref{thm1.2}. \hfill$\Box$

\section{Proof of Theorem \ref{thm1.3}}\label{sec5}

This section is devoted to proving Theorem \ref{thm1.3}.
To do so, we start with an algebraic decay rate of the total energy.
\begin{Lemma}\label{lemma5.1}
Under the assumptions of Theorem \ref{thm1.3}, there exists a positive constant $C$ independent of $t$ such that
\begin{align}\label{5.1}
\int\Big(\frac{1}{2}\rho|{\bf u}|^2+G(\rho)\Big)\mathrm{d}{\bf x}
\le C(1+t)^{-\frac{\gamma-1}{\gamma}}.
\end{align}
\end{Lemma}
\begin{proof}
Multiplying \eqref{1.1}$_1$ and \eqref{1.1}$_2$ by $G'(\rho)$ and ${\bf u}$, respectively, summing up and integrating the resultant over $\mathbb R^3$, we obtain that
\begin{align}\label{5.2}
\frac{\mathrm{d}}{\mathrm{d}t}\int\Big(\frac{1}{2}\rho|{\bf u}|^2+G(\rho)\Big)\mathrm{d}{\bf x}+\mu\norm{\nabla{\bf u}}_{L^2}^2+(\mu+\lambda)\norm{\nabla{\bf u}}_{L^2}^2=0,
\end{align}
which indicates that
\begin{align}\label{5.3}
\sup\limits_{0\le t<\infty}\int\Big(\frac{1}{2}\rho|{\bf u}|^2+G(\rho)\Big)\mathrm{d}{\bf x}+\int_0^\infty\big[\mu\norm{\nabla{\bf u}}_{L^2}^2+(\mu+\lambda)\norm{\nabla{\bf u}}_{L^2}^2\big]\mathrm{d}t
\le \int\Big(\rho_0|{\bf u}_0|^2+2G(\rho_0)\Big)\mathrm{d}{\bf x}.
\end{align}

Applying the operator $\Delta^{-1}\divv$ to \eqref{1.1}$_2$, we arrive at
\begin{align}\label{5.4}
P(\rho)=\partial_t(-\Delta)^{-1}\divv(\rho{\bf u})+(2\mu+\lambda)\divv{\bf u}-{\mathcal R}_i{\mathcal R}_j(\rho u^iu^j),
\end{align}
where ${\mathcal R}_i$ is the usual Riesz transform on $\mathbb R^3$: ${\mathcal R}_i=-(-\Delta)^{-\frac12}\partial_{x_i}$. Taking the $L^2$ product of \eqref{5.4} with $P(\rho)$, we infer that
\begin{align}\label{5.5}
\int P^2(\rho)\mathrm{d}{\bf x}=&\int\partial_t(-\Delta)^{-1}\divv(\rho{\bf u})P(\rho)\mathrm{d}{\bf x}+\int (2\mu+\lambda)\divv{\bf u}P(\rho)\mathrm{d}{\bf x}-\int {\mathcal R}_i{\mathcal R}_j(\rho u^iu^j)P(\rho)\mathrm{d}{\bf x}\notag\\\triangleq &~I_{1}+I_{2}+I_{3}.
\end{align}
We need to estimate each term $I_{i}\ (i=1,2,3)$. It follows from \eqref{1.1}$_{1}$ that
\begin{align*}
(P(\rho))_t+\divv[P(\rho){\bf u}]+(\gamma-1)P(\rho)\divv{\bf u}=0,
\end{align*}
which combined with \eqref{1.13}, \eqref{5.3}, the Marcinkiewicz multiplier theorem (see \cite[p. 96]{Stein}), \eqref{2.6}, H\"older inequality, and Cauchy--Schwarz inequality yields that
\begin{align}\label{5.7}
I_{1}=&\frac{\mathrm{d}}{\mathrm{d}t}\int(-\Delta)^{-1}\divv(\rho{\bf u})P(\rho)\mathrm{d}{\bf x}-\int(-\Delta)^{-1}\divv(\rho{\bf u})(P(\rho))_t\mathrm{d}{\bf x}\notag
\\=&\frac{\mathrm{d}}{\mathrm{d}t}\int (-\Delta)^{-1}\divv(\rho{\bf u})P(\rho)\mathrm{d}{\bf x}-\int \nabla(-\Delta)^{-1}\divv(\rho{\bf u})P(\rho){\bf u}\mathrm{d}{\bf x}\notag\\&+\int(-\Delta)^{-1}\divv(\rho{\bf u})(\gamma-1)P(\rho)\divv{\bf u}\mathrm{d}{\bf x}\notag\\ \le &\frac{\mathrm{d}}{\mathrm{d}t}\int (-\Delta)^{-1}\divv(\rho{\bf u})P(\rho)\mathrm{d}{\bf x}+C\big(\|\rho{\bf u}\|_{L^6}\|P(\rho)\|_{L^\frac{3}{2}}\|{\bf u}\|_{L^6}+\|\Delta^{-1}\divv(\rho{\bf u})\|_{L^6}\|P(\rho)\|_{L^3}\|\divv{\bf u}\|_{L^2}\big)\notag\\ \le &\frac{\mathrm{d}}{\mathrm{d}t}\int (-\Delta)^{-1}\divv(\rho{\bf u})P(\rho)\mathrm{d}{\bf x}+C\big(\|\rho\|_{L^\infty}\|P(\rho)\|_{L^\frac{3}{2}}\|{\bf u}\|_{L^6}^2+\|\rho{\bf u}\|_{L^2}\|P(\rho)\|_{L^3}\|\divv{\bf u}\|_{L^2}\big)\notag\\ \le &\frac{\mathrm{d}}{\mathrm{d}t}\int (-\Delta)^{-1}\divv(\rho{\bf u})P(\rho)\mathrm{d}{\bf x}+C\big(\|\rho\|_{L^\infty}\|P(\rho)\|_{L^\frac{3}{2}}\|{\bf u}\|_{L^6}^2 +\|\rho\|_{L^3}\|{\bf u}\|_{L^6}\|P(\rho)\|_{L^3}\|\divv{\bf u}\|_{L^2}\big)\notag\\ \le &\frac{\mathrm{d}}{\mathrm{d}t}\int_{\mathbb R^3}(-\Delta)^{-1}\divv(\rho{\bf u})P(\rho)\mathrm{d}{\bf x}+C\|\nabla{\bf u}\|^2_{L^2}.
\end{align}
Similarly, $I_{2}$ and $I_{3}$ can be handled as
\begin{align}\label{5.8}
&I_{2}\le \frac{1}{2}\|P(\rho)\|^2_{L^2}+C\|\nabla{\bf u}\|^2_{L^2},\\  \label{5.9}
&I_{3}\le C\|\rho u^iu^j\|_{L^3}\|P(\rho)\|_{L^\frac{3}{2}}\le C\|\rho\|_{L^\infty}\|{\bf u}\|_{L^6}^2\|P(\rho)\|_{L^\frac{3}{2}} \le C\|\nabla{\bf u}\|^2_{L^2}.
\end{align}
Hence, substituting \eqref{5.7}--\eqref{5.9} into \eqref{5.5} gives that
\begin{align}\label{5.10}
-\frac{\mathrm{d}}{\mathrm{d}t}\int (-\Delta)^{-1}\divv(\rho{\bf u})P(\rho)\mathrm{d}{\bf x}+\frac{1}{2}\|P(\rho)\|_{L^2}^2 \le C\|\nabla{\bf u}\|^2_{L^2}.
\end{align}

For $t\geq 0$, choose a positive constant $D_1$ and define the temporal energy functional
\begin{align*}
\mathcal M_1(t)=D_1\int\Big(\frac{1}{2}\rho|{\bf u}|^2\mathrm{d}{\bf x}+G(\rho)\Big)\mathrm{d}{\bf x}-\int(-\Delta)^{-1}\divv(\rho{\bf u})P(\rho)\mathrm{d}{\bf x}.
\end{align*}
By \eqref{1.8}, \eqref{5.3}, and $\|\rho\|_{L^1}=\|\rho_0\|_{L^1}$, we observe that
\begin{align*}
\int (-\Delta)^{-1}\divv(\rho{\bf u})P(\rho)\mathrm{d}{\bf x}\le &~C\|(-\Delta)^{-1}\divv(\rho{\bf u})\|_{L^6}\|P(\rho)\|_{L^\frac{6}{5}}\notag \\ \le&~C\|\rho{\bf u}\|_{L^2}\|P(\rho)\|_{L^\frac{6}{5}}\notag
\\ \le&~C\left(\|\sqrt{\rho}\|_{L^\infty}^6\|\sqrt{\rho}{\bf u}\|^6_{L^2}+\|P(\rho)\|_{L^1}\|P(\rho)\|_{L^\infty}^\frac{1}{5}\right)\notag \\ \le&~C\int\Big(\frac{1}{2}\rho|{\bf u}|^2+G(\rho)\Big)\mathrm{d}{\bf x},
\end{align*}
and
\begin{align}\label{5.11}
\int\Big(\frac{1}{2}\rho|{\bf u}|^2+G(\rho)\Big)\mathrm{d}{\bf x} &=\frac{1}{2}\int \rho^\frac{\gamma-1}{2\gamma-1}(\sqrt\rho|{\bf u}|)^\frac{2\gamma}{2\gamma-1}|{\bf u}|^\frac{2(\gamma-1)}{2\gamma-1}\mathrm{d}{\bf x}+\int G(\rho)\mathrm{d}{\bf x}\notag\\ &\le C\|\rho\|_{L^\frac{3}{2}}^\frac{\gamma-1}{2\gamma-1}\|\sqrt{\rho}{\bf u}\|_{L^2}^\frac{2\gamma}{2\gamma-1}\|{\bf u}\|_{L^6}^\frac{2(\gamma-1)}{2\gamma-1}
+C\|\rho\|_{L^1}^\frac{\gamma}{2\gamma-1}\|\rho\|_{L^{2\gamma}}
^\frac{2\gamma(\gamma-1)}{2\gamma-1}
\notag \\ &\le C\|(P(\rho),\nabla {\bf u})\|^\frac{2(\gamma-1)}{2\gamma-1}_{L^2}.
\end{align}
Thus, if $D_1$ is chosen large enough, one has that
\begin{align}
\mathcal M_1(t)\sim\int\Big(\frac{1}{2}\rho|{\bf u}|^2+G(\rho)\Big)\mathrm{d}{\bf x}\nonumber.
\end{align}

Taking a linear combination of \eqref{5.2} and \eqref{5.10}, and using \eqref{5.11}, we see that there exists a positive constant $C$ independent of $t$ such that
\begin{align*}
\frac{\mathrm{d}}{\mathrm{d}t}\mathcal M_1(t)+C(\mathcal M_1(t))^\frac{2\gamma-1}{\gamma-1}\le 0.
\end{align*}
Solving the above inequality directly yields that
\begin{align*}
\mathcal M_1(t)\le C(1+t)^{-\frac{\gamma-1}{\gamma}},
\end{align*}
as desired \eqref{5.1}.
\end{proof}

For $\delta>0$, it infers from \eqref{5.1} and \eqref{5.3} that there exists a positive constant $T_{10}$ such that
\begin{align}\label{5.16}
\int\Big(\frac{1}{2}\rho|{\bf u}|^2+G(\rho)\Big)(T_{10})\mathrm{d}{\bf x}+\norm{\nabla{\bf u}(T_{10})}_{L^2}\le \delta.
\end{align}
The following Proposition \ref{proposition5.2} (see \cite[Theorem 1.3]{LiXin}) deals with the large-time asymptotic behavior of global classical solutions.
\begin{Proposition}\label{proposition5.2}
Under the assumptions of Theorem \ref{thm1.3}, it holds that
\begin{align}
\sup\limits_{T_{10}\le t<\infty}\big[t\big(\norm{\nabla{\bf u}}_{L^2}^2+\norm{P}_{L^2}^2\big)\big]+\int_{T_{10}}^\infty t\big(\|\sqrt{\rho}\dot{\bf u}\|_{L^2}^2+\|P\|_{L^3}^3\big)\mathrm{d}t\le C,\label{5.13}
\\
\sup\limits_{T_{10}\le t<\infty}\big[t^2\big(\|\sqrt{\rho}\dot{\bf u}\|_{L^2}^2+\|P\|_{L^3}^3\big)\big]+\int_{T_{10}}^\infty t^2\big(\norm{\nabla{\dot{\bf u}}}_{L^2}^2+\norm{P}^4_{L^4}\big)\mathrm{d}t\le C.\label{5.14}
\end{align}
\end{Proposition}

Now we are in a position to prove Theorem \ref{thm1.3}.

{\bf Proof of Theorem \ref{thm1.3}}.
By \eqref{1.1}$_1$ and the definition of  $F$ in \eqref{1.4}, one gets that
\begin{align}\label{5.15}
\frac{\mathrm{d}}{\mathrm{d}t}\rho+\frac{1}{2\mu+\lambda}\rho^{\gamma+1}
=-\frac{1}{2\mu+\lambda}\rho F,
\end{align}
Multiplying \eqref{5.15} by $\gamma t^2 \rho^{\gamma-1}$ indicates that
\begin{align*}
\frac{\mathrm{d}}{\mathrm{d}t}(t^2\rho^\gamma)+\frac{\gamma}{2\mu+\lambda}t^2\rho^{2\gamma}
=2t\rho^\gamma-\frac{\gamma}{2\mu+\lambda}t^2\rho^\gamma F,
\end{align*}
which implies that
\begin{align}\label{wz4}
t^2\rho^\gamma(t)+\frac{\gamma}{2\mu+\lambda}\int_{T_{10}}^t\tau^2\rho^{2\gamma}\mathrm{d}\tau= T_{10}^2\rho^\gamma(T_{10})+ 2\int_{T_{10}}^t\tau\rho^{\gamma}\mathrm{d}\tau-\frac{\gamma}{2\mu+\lambda}\int_{T_{10}}^t \tau^2\rho^\gamma F\mathrm{d}\tau.
\end{align}

By Cauchy--Schwarz inequality, we have
\begin{align}\label{5.17}
\int_{T_{10}}^t\tau\rho^{\gamma}\mathrm{d}\tau \le \frac{\gamma}{2(2\mu+\lambda)}\int_{T_{10}}^t\tau^2\rho^{2\gamma}\mathrm{d}\tau +C\int_{T_{10}}^t \mathrm{d}\tau\le \frac{\gamma}{2(2\mu+\lambda)}\int_{T_{10}}^t\tau^2\rho^{2\gamma}\mathrm{d}\tau +C t.
\end{align}
It follows from \eqref{2.2}, \eqref{2.3}, \eqref{2.6}, \eqref{5.13}, and \eqref{5.14} that
\begin{align}
\int_{T_{10}}^t \tau^2\rho^\gamma F\mathrm{d}\tau &\le  \frac{\gamma}{2(2\mu+\lambda)}\int_{T_{10}}^t\tau^2\rho^{2\gamma}\mathrm{d}\tau+C\int_{T_{10}}^t \tau^2\norm{F}_{L^\infty}^2\mathrm{d}t\notag
\\ &\le \frac{\gamma}{2(2\mu+\lambda)}\int_{T_{10}}^t\tau^2\rho^{2\gamma}\mathrm{d}\tau+\int_{T_{10}}^t \tau^2 \norm{F}_{L^6}\norm{\nabla F}_{L^6}\mathrm{d}t\notag
\\ &\le \frac{\gamma}{2(2\mu+\lambda)}\int_{T_{10}}^t\tau^2\rho^{2\gamma}\mathrm{d}\tau+C\int_{T_{10}}^t \tau^2\norm{\rho\dot{\bf u}}_{L^2}\norm{\rho\dot{\bf u}}_{L^6}\mathrm{d}t\notag
\\ &\le\frac{\gamma}{2(2\mu+\lambda)}\int_{T_{10}}^t\tau^2\rho^{2\gamma}\mathrm{d}\tau+ C\int_{T_{10}}^t \tau \norm{\nabla\dot{\bf u}}_{L^2}\mathrm{d}t\notag
\\ &\le \frac{\gamma}{2(2\mu+\lambda)}\int_{T_{10}}^t\tau^2\rho^{2\gamma}\mathrm{d}\tau+Ct.\label{5.18}
\end{align}
Substituting \eqref{5.17} and \eqref{5.18} into \eqref{wz4}, we obtain that
\begin{align*}
\norm{P(\cdot,t)}_{L^\infty}\le C(1+t)^{-1},
\end{align*}
which together with \eqref{5.1} and H\"older inequality leads to
\begin{align*}
\norm{P(\cdot,t)}_{L^p}\le\norm{P(\cdot,t)}_{L^1}^\frac{1}{p}\norm{P(\cdot,t)}_{L^\infty}^{1-\frac{1}{p}}\le C(1+t)^{-1+\frac{1}{p\gamma}},\ \ \text{for any}\ 1\leq p\leq \infty,
\end{align*}
as desired \eqref{1.16}.

Next, one infers from \eqref{1.16}, \eqref{5.13}, and H\"older inequality that
\begin{align*}
\int \rho|{\bf u}|^2\mathrm{d}{\bf x}
\le \norm{\rho}_{L^\frac{3}{2}}\norm{\bf u}_{L^6}^2
\le
C(1+t)^{-1}\norm{\rho}_{L^1}^{\frac{2}{3}}
\norm{\rho }_{L^\infty}^{\frac{1}{3}}\le C(1+t)^{-1-\frac{1}{3\gamma}},
\end{align*}
as desired \eqref{1.17}. \hfill$\Box$

\section*{Conflict of interest}
The authors declare that they have no conflict of interest.
	
\section*{Data availability}
Data sharing is not applicable to this article as no new data were created or analyzed in this study.

\end{document}